%% file: ssptsrk.tex
\documentclass[final]{siamltex}
\usepackage{amsmath}
\usepackage{amsfonts}
\usepackage[dvips,final]{graphicx}
\usepackage{graphicx}
\usepackage{psfrag}
\usepackage{pstricks,pst-node,pst-tree}
\usepackage{dcolumn}
\usepackage{rotating}
\usepackage{xspace}
\usepackage{url}
\usepackage{upgreek}
\usepackage{multicol}
\usepackage{pict2e,slashbox,multirow}
\usepackage{verbatim}
\usepackage{url}
\usepackage{tikz}
\usepackage[colorlinks=true,citecolor=blue,urlcolor=blue]{hyperref}
\title{Strong Stability Preserving Two-step Runge--Kutta Methods}

\author{%
  David I. Ketcheson\thanks{4700 King Abdullah University of Science \&
    Technology, Thuwal 23955, Saudi Arabia.
   (\protect\url{david.ketcheson@kaust.edu.sa}).
    The work of this author was supported by a
    U.S.  Dept. of Energy Computational Science Graduate Fellowship and
    by funding from King Abdullah University of Science and Technology (KAUST).}, 
    \and
  Sigal Gottlieb\thanks{Department of Mathematics, University of
    Massachusetts Dartmouth, North Dartmouth, MA 02747
    (\protect\url{sgottlieb@umassd.edu}). This work was
    supported by AFOSR grant number FA9550-09-1-0208.},
    \and
  Colin B. Macdonald\thanks{Mathematical Institute, University of
    Oxford, Oxford, OX1\,3LB, UK
    (\protect\url{macdonald@maths.ox.ac.uk}).  The work of this author
    was supported by an NSERC postdoctoral fellowship, NSF grant
    number CCF-0321917, and by Award No KUK-C1-013-04 made by King
    Abdullah University of Science and Technology (KAUST).}%
}
\begin{document}
\maketitle

%\begin{keywords}
%\end{keywords}

\bibliographystyle{siam}

\begin{abstract} 
We investigate the strong stability preserving (SSP) property of 
two-step Runge--Kutta (TSRK) methods.  We prove that all SSP
TSRK methods belong to a particularly simple subclass of TSRK methods, 
in which stages
from the previous step are not used. We derive simple order conditions 
for this subclass.  Whereas explicit SSP Runge--Kutta
methods have order at  most four, we prove that explicit SSP TSRK methods
have order at most eight.  We present TSRK methods of up to
eighth order that were found by numerical search.  These
methods have larger SSP coefficients than any known methods of
the same order of accuracy, and may be implemented in a form 
with relatively modest storage requirements.
The usefulness of the TSRK methods is demonstrated through
numerical examples, including integration of very high order WENO
discretizations.
\end{abstract}

%\pagestyle{myheadings}
%\thispagestyle{plain}
%\markboth{D. I. Ketcheson}{Optimal SSP GLMs}
%============================================================
%Define macros, environments, etc.
\input{macros}
%============================================================

\section{Strong Stability Preserving Methods}
\label{sec:intro}
  \input{intro}

\section{SSP Two-step Runge--Kutta Methods\label{sect:sspglms}}
  \input{sspglms}

\section{Order Conditions and a Barrier}
  \input{tsrk_ocs}

\section{Optimal SSP Two-step Runge--Kutta methods} % for nonlinear problems}
  \input{twostep}

%\section{Optimal SSP Multistep Runge--Kutta methods for linear problems}
%  \input{restrict}
  
  \section{Numerical Experiments}
  \input{numerics}

  \section{Conclusions}
  \input{conclusions}

\noindent {\bf Acknowledgment.} The authors are grateful to an anonymous referee, 
whose careful reading
and detailed comments improved several technical details of the paper.
%\section{Explicit SSP GLMs for linear problems}
%  \input{explinglm}

%\section{SSP GLMs with downwinded operators}
%  \input{downwind}

\appendix
\section{Coefficients of Numerically Optimal Methods}
\input tsrk_coefficients.tex

\bibliography{ssptsrk}

\end{document}

%% file: macros.tex
\newtheorem{thm}{Theorem}
\newtheorem{dfn}{Definition}
\newtheorem{lem}{Lemma}
\newtheorem{cor}{Corollary}
\newtheorem{rmk}{Remark}
\newcommand{\qh}{\hat{q}}
\newcommand{\be}{\begin{equation}}
\newcommand{\ee}{\end{equation}}
\newcommand{\bq}{\mathbf{q}}
\newcommand{\bx}{\mathbf{x}}
\newcommand{\by}{\mathbf{y}}
\newcommand{\br}{\mathbf{r}}
\newcommand{\imh}{{i-\frac{1}{2}}}
\newcommand{\iph}{{i+\frac{1}{2}}}
\newcommand{\ipmh}{{i \pm \frac{1}{2}}}
\newcommand{\jph}{{j+\frac{1}{2}}}
\newcommand{\Aop}{{\cal A}}
\newcommand{\Bop}{{\cal B}}
\newcommand{\Wop}{{\cal W}}
\newcommand{\Oop}{{\cal O}}
\newcommand{\DQ}{\Delta Q}
\newcommand{\Dq}{\Delta q}
\newcommand{\Dx}{\Delta x}
\newcommand{\Dy}{\Delta y}
\newcommand{\Du}{\Delta u}
\newcommand{\bu}{\mathbf{u}}
\newcommand{\bv}{\mathbf{v}}
\newcommand{\bw}{\mathbf{w}}
\newcommand{\bU}{\mathbf{U}}
\newcommand{\bV}{\mathbf{V}}
\newcommand{\bF}{\mathbf{F}}
\newcommand{\Lop}{{\cal L}}
\newcommand{\Sop}{{\cal S}}
\newcommand{\Fop}{{\cal F}}
\newcommand{\Dofr}{{\cal D}(r)}
\newcommand{\Dt}{\Delta t}
\newcommand{\bbA}{\mathbf{A}}
\newcommand{\bbZ}{\mathbf{Z}}
\newcommand{\bbK}{\mathbf{K}}
\newcommand{\bbI}{\mathbf{I}}
\newcommand{\bbb}{\bar{\mathbf{b}}}
\newcommand{\bbe}{\mathbf{e}}
\newcommand{\bbone}{\mathbf{1}}
\newcommand{\lnorm}{\left\|}
\newcommand{\rnorm}{\right\|}

\newcommand{\dx}{\Delta x}
\newcommand{\dt}{\Delta t}
\newcommand{\aij}{\alpha_{ij}}
\newcommand{\bij}{\beta_{ij}}
\newcommand{\lt}{\tilde{L}}

\newcommand{\hf}{\frac{1}{2}}
\def\half{\frac{1}{2}}
\newcommand{\fracStrut}{\rule[-1.0ex]{0pt}{3.1ex}}
\newcommand{\hfs}{\ensuremath{\frac{1}{2}}\fracStrut}
\newcommand{\scinot}[2]{\ensuremath{#1\times10^{#2}}}
\newcommand{\dee}{\mathrm{d}}
\newcommand{\dye}{\partial}
\newcommand{\diff}[2]{\frac{\dee #1}{\dee #2}}
\newcommand{\pdiff}[2]{\frac{\dye #1}{\dye #2}}
\newcommand{\Real}{\mathbb{R}}
\newcommand{\Complex}{\mathbb{C}}
% matrices
\newcommand{\m}[1]{\mathbf{#1}}
\newcommand{\mA}{\m{A}}
\newcommand{\mAb}{\bar{\m{A}}}
\newcommand{\vah}{\hat{\v{a}}}
\newcommand{\mAh}{\hat{\m{A}}}
\newcommand{\mD}{\m{D}}
\newcommand{\mS}{\m{S}}
\newcommand{\mT}{\m{T}}
\newcommand{\mR}{\m{R}}
\newcommand{\mP}{\m{P}}
\newcommand{\mM}{\m{M}}
\newcommand{\mQ}{\m{Q}}
\newcommand{\mI}{\m{I}}
\newcommand{\mK}{\m{K}}
\newcommand{\mL}{\m{L}}
\newcommand{\mzero}{\m{0}}
% these use the upgreek package to get non-italic greek, which doesn't
% seem to work with \mathbf so these have to be setup manually to
% match \m
\newcommand{\matalpha}{\boldsymbol{\upalpha}}
\newcommand{\matbeta}{\boldsymbol{\upbeta}}
\newcommand{\mattheta}{\boldsymbol{\uptheta}}
\newcommand{\mateta}{\boldsymbol{\upeta}}
\newcommand{\matmu}{\boldsymbol{\upmu}}
\newcommand{\matlambda}{\boldsymbol{\uplambda}}
\newcommand{\matgamma}{\boldsymbol{\upgamma}}
\newcommand{\matdelta}{\boldsymbol{\updelta}}
% vectors
\renewcommand{\v}[1]{\boldsymbol{#1}}
\newcommand{\transpose}{^\mathrm{T}}
\newcommand{\thT}{\mattheta\transpose}
\newcommand{\bT}{\v{b}\transpose}
\newcommand{\bh}{\hat{b}}
\newcommand{\vbh}{\hat{\v{b}}}
\newcommand{\bhT}{\hat{\v{b}}\transpose}
\newcommand{\vb}{\v{b}}
\newcommand{\vc}{\v{c}}
\newcommand{\ve}{\v{e}}
\newcommand{\vu}{\v{u}}
\newcommand{\vl}{\v{l}}
\newcommand{\vv}{\v{v}}
\newcommand{\vy}{\v{y}}
\newcommand{\vd}{\v{d}}
\newcommand{\va}{\v{a}}
\newcommand{\vf}{\v{f}}
\newcommand{\Matlab}{{\sc Matlab}\xspace}
\newcommand{\BARON}{{\sc Baron}\xspace}
\newcommand{\code}[1]{\textsf{#1}}
% the SSP coefficient
\newcommand{\sspcoef}{\mathcal{C}}
\newcommand{\clin}{\sspcoef_{\textup{lin}}}
\newcommand{\ceff}{\sspcoef_{\textup{eff}}}
\newcommand{\DtFE}{\Dt_{\textup{FE}}}

\newcommand{\Inc}{\textrm{Inc}}

\newcommand{\ty}{\tilde{y}}
\newcommand{\tu}{\tilde{u}}
\newcommand{\pt}{\tilde{p}}
\newcommand{\lte}{\tau}
\newcommand{\ste}{\boldsymbol{\tau}}
\newcommand{\bgamma}{\boldsymbol{\gamma}}
\newcommand{\btheta}{\boldsymbol{\theta}}
\newcommand{\bty}{\mathbf{\ty}}
\newcommand{\btu}{\mathbf{\tilde{u}}}

\newcommand{\gerr}{\epsilon}
\newcommand{\gserr}{\boldsymbol{\epsilon}}
\newcommand{\gsteperr}{\bar{\boldsymbol{\epsilon}}}
\newcommand{\rhserr}{\delta}
\newcommand{\rhsserr}{\boldsymbol{\delta}}

\newcommand{\bff}{\mathbf{f}}
\newcommand{\bFf}{\mathbf{F}}
\newcommand{\bfb}{\mathbf{f}_u}
\newcommand{\bfbt}{\tilde{\mathbf{f}}_u}
\newcommand{\btf}{\mathbf{\tilde{f}}}
\newcommand{\bone}{\mathbf{1}}
\newcommand{\bb}{\mathbf{b}}
\newcommand{\bbh}{\hat{\mathbf{b}}}
\newcommand{\bc}{\mathbf{c}}
\newcommand{\bt}{\mathbf{t}}
\newcommand{\bg}{\mathbf{g}}
\newcommand{\bfe}{\mathbf{e}}
\newcommand{\mC}{\m{C}}

\renewcommand{\v}[1]{\mathbf{#1}}

%% file: intro.tex
%\subsection{SSP Runge--Kutta and multi-step methods}
The concept of strong stability preserving methods  was first introduced by Shu and Osher in 
\cite{shu1988} for use with total variation diminishing spatial discretizations of a hyperbolic conservation law:
\begin{eqnarray*}
U_t +f(U)_x = 0.
\end{eqnarray*}
When the spatial derivative is discretized, we obtain the  system of ODEs 
\begin{eqnarray}
\label{ode}
u_t = F(u),
\end{eqnarray}
where $u$ is a vector of approximations to $U$,  $u_j \approx U(x_j) $. 
The spatial discretization is carefully designed so that when this ODE
is fully discretized using the forward Euler  method, certain convex functional properties
(such as the  total variation) of the numerical solution do not increase,
\begin{eqnarray} \label{monotonicity}
\| u^n + \dt F(u^{n}) \| \leq \| u^n \|
\end{eqnarray}
for all small enough step sizes $\dt \leq \DtFE$.  Typically, we need methods of  higher order 
and we wish to guarantee that the higher-order time discretizations will preserve this strong stability 
property.
This guarantee is obtained by observing that if a time discretization can be decomposed into convex combinations
of forward Euler steps, then any convex functional  property
(referred to herein as a {\em strong stability} property)
satisfied by forward Euler will be {\em preserved}
by the higher-order time discretizations, perhaps under a different time-step restriction.

%To build {\em strong stability preserving methods}
% we {\em assume} that the spatial discretization 
%satisfies some strong stability property under forward Euler time-stepping
%\begin{eqnarray}
%\| u + \dt F(u) \| \leq \| u \|
%\end{eqnarray}
%for any timestep $\dt \leq \DtFE$,
%and  search for higher-order methods which {\em preserve} this strong stability property,
%under the timestep restriction $\dt \leq \sspcoef \DtFE$, for some $\sspcoef>0$.
%
%In other words, 
Given a semi-discretization of the form \eqref{ode} and convex  functional 
$\|\cdot\|$, we assume that there exists a value $\DtFE$ such that, for
all $u$,
\be \label{FEcond} \|u+\dt F(u)\|\leq \|u\| \mbox{   for } 0\le \Dt \le \DtFE.\ee
A $k$-step numerical method for \eqref{ode} computes 
the next solution value $u^{n+1}$ from previous values
$u^{n-k+1},\dots,u^n$.
We say that the method is {\em strong stability preserving} (SSP) 
if (in the solution of \eqref{ode}) it holds that
\be \label{kstepmonotonicity}
\|u^{n+1}\|\le\max\left\{\|u^n\|,\|u^{n-1}\|,\dots,\|u^{n-k+1}\|\right\}.
\ee
whenever \eqref{FEcond} holds and the timestep satisfies
\be \label{tstepcond} \Dt\le \sspcoef \DtFE. \ee
Throughout this work, $\sspcoef$ is taken to be
the largest value such that \eqref{tstepcond} and \eqref{FEcond} 
together always imply \eqref{kstepmonotonicity}.
This value $\sspcoef$ is called the {\em SSP coefficient} of the method.

For example, consider explicit multistep methods  \cite{shu1988b}:
\begin{eqnarray} \label{lmmSO}
u^{n+1} & = & \sum_{i=1}^{k} \left( \alpha_i u^{n+1-i} +
\dt \beta_i F(u^{n+1-i}) \right).% \qquad \alpha_i \geq 0. 
\end{eqnarray}  
Since $\sum_{i=1}^{k} \alpha_i =1$ for any consistent method, any such method
can be written as convex combinations of forward Euler steps if all the 
coefficients are non-negative:
\begin{eqnarray*}
u^{n+1} & = & \sum_{i=1}^{k} \alpha_i  \left( u^{n+1-i} +
\frac{\beta_i}{\alpha_i} \dt  F(u^{n+1-i}) \right).% \qquad \alpha_i \geq 0. 
\end{eqnarray*}  
If the forward Euler method applied to \eqref{ode}
is strongly stable under the timestep restriction $\Dt\le\DtFE$ and
$\alpha_i,\beta_i\ge0$  then the solution obtained by the
 multistep method (\ref{lmmSO})
satisfies the strong stability bound \eqref{kstepmonotonicity} 
under the timestep restriction
\begin{equation*}
\dt \leq  \min_{i} \frac{\alpha_i}{\beta_i} \DtFE, 
\end{equation*}
(if any of the $\beta$'s are equal to zero, the corresponding 
ratios are considered infinite).

In the case of a one-step method the monotonicity requirement \eqref{kstepmonotonicity} reduces to
$$ \|u^{n+1}\|\le\|u^n\|.  $$
For example,  an $s$-stage explicit Runge--Kutta method is written in the form \cite{shu1988},
\begin{eqnarray}
\label{rkSO}
u^{(0)} & =  & u^n, \nonumber \\
u^{(i)} & = & \sum_{j=0}^{i-1} \left( \aij u^{(j)} +
\dt \bij F(u^{(j)}) \right), \\% \quad \aik \geq 0,  \qquad i=1 ,..., m \\
 u^{n+1} & = & u^{(s)}. \nonumber
\end{eqnarray}
If all the coefficients are non-negative,
%and  $\sum_{j=0}^{i-1} \aij =1$ (for consistency),
%the forms (\ref{rkSO}) and (\ref{lmmSO}) 
each stage of   the Runge--Kutta method can be rearranged into convex 
combinations of forward Euler steps, with a modified step size:
\begin{eqnarray*}
\| u^{(i)}\| & =  & 
\| \sum_{j=0}^{i-1} \left( \aij u^{(j)} + \dt \bij F(u^{(j)}) \right) \|   \nonumber \\
& \leq &  \sum_{j=0}^{i-1} \aij  \, \left\| u^{(j)} + \dt \frac{\bij}{\aij} F(u^{(j}) \right\| .  
\end{eqnarray*}
Now, since each $\| u^{(j)} + \dt \frac{\bij}{\aij} F(u^{(j)}) \| \leq \| u^{(j)} \|$ as 
long as 
$ \frac{\bij}{\aij}  \dt \leq \DtFE$, and since  $\sum_{j=0}^{i-1} \aij =1$ by consistency,
we have $ \| u^{n+1}\| \leq \| u^{n}\| $ as long as  $ \frac{\bij}{\aij} \dt \leq  \DtFE$ 
for all $i$ and $j$.
Thus, if the forward Euler method applied to \eqref{ode}
is strongly stable under the timestep restriction $\Dt\le\DtFE$, {\em i.e.} 
\eqref{FEcond} holds, and if $\aij,\bij\ge0$ 
then the solution obtained by the Runge--Kutta method \eqref{rkSO} 
satisfies the strong stability bound \eqref{monotonicity} 
under the timestep restriction
\begin{equation*}
\dt \leq \min_{i,j} \frac{\aij}{\bij} \DtFE.
\end{equation*}
As above, if any of the $\beta$'s are equal to zero, the corresponding 
ratios are considered infinite.

This approach can easily be generalized to implicit Runge--Kutta methods and
implicit linear multistep methods.  Thus it provides
sufficient conditions for strong stability of high-order explicit and implicit
Runge--Kutta and multistep methods. In fact, it can be shown 
from the connections between
SSP theory and contractivity theory \cite{ferracina2004,
ferracina2005,higueras2004a, higueras2005a}
that these conditions are not only sufficient, they are necessary as well. 

Research in the field of SSP methods focuses on finding high-order time discretizations
with the largest allowable time-step.  Unfortunately, 
explicit SSP Runge--Kutta methods with positive coefficients cannot be more than
fourth-order accurate \cite{kraaijevanger1991,ruuth2001}, 
and  explicit SSP linear
multistep methods of high-order accuracy require very many steps in order to
have reasonable timestep restrictions.  For instance, to obtain a fifth-order explicit linear multistep method with 
a time-step restriction of $\dt \le 0.2 \DtFE$
 requires nine steps; for a sixth-order method,
this increases to thirteen steps \cite{lenferink1989}. In practice, the large storage requirements of these
 methods make them unsuitable for
the solution of the large  systems of ODEs resulting from semi-discretization of a PDE.
Multistep methods with larger SSP coefficients and fewer stages have been obtained by considering 
special starting procedures \cite{hundsdorfer2005,ruuth2005}.  

%However,
%the resulting methods of fifth and higher order perform poorly in
%practice \cite{ruuth2005} {\bf DK I may have written this last statement,
%but I think it is too strong and not precise enough.  Also, future work
%could include applying Hundsdorfer \& Ruuth's approach to GLMs}.

Because of the lack of practical explicit SSP methods of very high order, high-order
spatial discretizations for hyperbolic PDEs are often paired with lower-order time discretizations; some examples of this include
\cite{carrillo2003,cheng2003,cheruvu2007,enright2002,feng2004,jin2005,
labrunie2004,peng1999,tanguay2003}.  This may lead 
to loss of accuracy, particularly for long time simulations.
In an extreme case \cite{gerolymos2009},
WENO schemes of up to $17$th-order 
were paired with third-order SSP Runge--Kutta time integration; of course,
convergence tests indicated only third-order convergence for the
fully discrete schemes.
Practical higher-order accurate SSP time discretization methods are needed for 
the time evolution of ODEs resulting from high-order spatial discretizations.

To obtain higher-order explicit SSP time discretizations, methods that 
include both multiple steps and multiple stages have been considered.
These methods are a subclass of explicit general linear methods that 
allow higher order with positive  SSP coefficients.
Gottlieb et.~al.~considered a class of two-step, two-stage methods \cite{gottlieb2001}.
Another class of such methods was considered by Spijker \cite{spijker2007}.
Huang \cite{huang2009} considered hybrid methods with many steps, and
 found methods of up to seventh-order (with seven steps) with reasonable SSP coefficients.
 Constantinescu and Sandu \cite{constantinescu2009} considered two- and three-step Runge--Kutta methods,
 with a focus on finding SSP methods with stage order up to four.

In this work we consider a class of two-step multi-stage Runge--Kutta methods, which are a generalization of both linear 
multistep methods and Runge--Kutta methods.  We have found that deriving 
the order conditions using a generalization of the approach presented in 
\cite{albrecht1996},
and formulating the optimization problem using the approach from
\cite{ketcheson2009} allows us to efficiently find methods of up to eighth 
order with relatively modest storage requirements and large effective SSP coefficient.
We also report optimal lower-order methods; our results agree with those of \cite{constantinescu2009}
for second, third, and fourth-order methods of up to four stages, and improve upon other
methods previously found both in terms of order and the size of the SSP coefficient.

The major result of this paper is the development of SSP two-step Runge--Kutta 
methods 
of up to eighth order that are efficient and practical.  In Section 2,
we discuss some classes of  two-step Runge--Kutta (TSRK) methods and prove that all
SSP TSRK methods belong to one of two simple subclasses.
In Section 3, we derive order conditions and show
that explicit SSP TSRK methods have order at most eight.
In Section 4, we formulate the optimization problem, give an efficient form 
for the implementation of SSP two-step Runge--Kutta methods, 
and present optimal methods of up to eighth order. The properties of our
methods are compared with those of
existing SSP methods including Runge--Kutta, linear multi-step, 
and hybrid methods \cite{huang2009}, as well as the two- and three-step methods in 
 \cite{constantinescu2009}.  Numerical verification of the optimal methods 
 and  a demonstration of the need for high-order time discretizations for use with 
 high-order spatial discretizations is presented in Section 5.  
 Conclusions and future work are discussed in Section 6.

%% file: sspglms.tex
%=========================================================
%\subsection{Two-step Runge--Kutta Methods}
%\textbf{\textit{I removed a redundant subsection header}}
%=========================================================
The principal focus of this work is on the strong stability preserving
properties of two-step Runge--Kutta (TSRK) methods.
A general class of TSRK methods was studied in 
\cite{jackiewicz1995,butcher1997,hairer1997,verner2006}.
TSRK methods are a generalization of Runge--Kutta methods that include 
values and stages from the previous step:
\begin{subequations} \label{tsrk_JT}
\begin{align}
y_i^n &= d_i u^{n-1} + (1-d_i) u^n  + \Dt\sum_{j=1}^{s} \hat{a}_{ij} F(y_j^{n-1})
 + \Dt\sum_{j=1}^{s} a_{ij} F(y_j^n), &  1 \le i \le s,\\
u^{n+1} &= \theta u^{n-1} + (1-\theta) u^n  +\Dt\sum_{j=1}^{s} \bh_{j} F(y_j^{n-1})
 + \Dt\sum_{j=1}^{s} b_j F(y_j^n).
\end{align}
\end{subequations}
Here $u^n$ and $u^{n-1}$ denote solution values at the times $t=n\Dt$
and $t=(n-1)\Dt$, while the values $y^n_i$ are intermediate stages
used to compute the solution at the next time step.  We will use the
 matrices and vectors $\mA$, $\mAh$, $\vb$, $\vbh$, and
$\vd$ to refer to the coefficients of the method.

We are interested only in TSRK methods that have the strong stability preserving
property.  As we will prove in Theorem~\ref{thm:allssptsrk}, this greatly reduces
the set of methods relevant to our study.  Except in special cases, the method
\eqref{tsrk_JT} cannot be strong stability preserving unless all
of the coefficients $\hat{a}_{ij},\bh_j$ are identically zero.
A brief explanation of this requirement is as follows.
Since method \eqref{tsrk_JT} does not include terms of the form
$y_i^{n-1}$, it is not possible to write
a stage of method \eqref{tsrk_JT} as a convex
combination of forward Euler steps if the stage includes terms of the
form $F(y_i^{n-1})$.  This is because those stages depend on $u^{n-2}$,
which is not available in a two-step method.
%We state and prove this result more
%precisely in Section \ref{sect:reducibility}.
%If terms of the form $F(y_j^{n-l})$ (for $1\le j\le s, 1\le l \le k$) are included
%in these methods, it is straightforward to deduce from Lemma
%\ref{lem:inc} that their coefficients must vanish (unless terms
%of the form $y_j^{n-l}$ are included also).  

Hence we are led to consider simpler methods of the following form
(compare \cite[p. 362]{Hairer:ODEs2}).  We call these {\bf Type I} methods:
\begin{subequations}
\label{tsrk_b}
\begin{align}
%k-step version:
%y_i^n &= \sum_{j=1}^{k} d_{ij} u^{n+1-j} 
% + \Dt\sum_{j=1}^{s} a_{ij} F(y_j^n), &  1 \le i \le s,\\
%u^{n+1} &= \sum_{j=1}^{k} \theta_j u^{n+1-j} 
% + \Dt\sum_{j=1}^{s} b_j F(y_j^n).
y_i^n &= d_i u^{n-1} + (1-d_i) u^n
 + \Dt\sum_{j=1}^{s} a_{ij} F(y_j^n), &  1 \le i \le s,\\
u^{n+1} &= \theta u^{n-1} + (1-\theta) u^n 
 + \Dt\sum_{j=1}^{s} b_j F(y_j^n).
\end{align}
\end{subequations}

Now consider the special case in which the method \eqref{tsrk_JT} has some stage
$y_i^n$ identically equal to $u^n$.
%{\bf cbm: how about ``has some stage, say the second stage, $y_2^n$ identically equal to $u^n$''?, assuming we make my change below}
Then including terms proportional to $F(u^n)$ will not prevent
the method from being written as a convex combination of forward Euler
steps; furthermore, since $y_i^{n-1}=u^{n-1}$, terms of the form $F(u^{n-1})$
can also be included.  This leads to what we will call {\bf Type II}
methods, which have the form:
\begin{subequations} \label{tsrk_c}
\begin{align}
y_1^n &= u^n, \\
y_i^n &= d_i u^{n-1} + (1-d_i) u^n 
 + \hat{a}_{i} \Dt F(u^{n-1}) 
 + \Dt\sum_{j=1}^{s} a_{ij} F(y_j^n), &  2 \le i \le s,\\
u^{n+1} &= \theta u^{n-1} + (1-\theta) u^n 
 + \hat{b}_{1} \Dt F(u^{n-1}) 
 + \Dt\sum_{j=1}^{s} b_j F(y_j^n).
\end{align}
\end{subequations}
Here we have assumed that the first stage is the one equal to $u^n$,
which involves no loss of generality.  We can refer to the
coefficients of Type II methods in the matrix/vector notation of
\eqref{tsrk_JT} except that matrix $\mAh$ reduces to vector $\vah$ and
we have $d_1=\hat{a}_{1}=0$ and $a_{1j}=0$ for all $1 \le j \le s$.
%We can refer to the coefficients of Type II methods in a similar
%manner to those of Type I methods by setting $d_1=\hat{a}_{1}=0$ and 
%$a_{1j}=0$ for all $1 \le j \le s$.

\begin{rmk} \label{rmk1}
From a theoretical perspective, the distinction between Type I and
Type II methods may seem artificial, since
the class of all Type I methods is equivalent to
the class of all Type II methods.
From a practical perspective, however, the distinction is very
useful.  Transforming a given method from one type
to the other generally requires adding a stage.
Thus the class of $s$-stage Type I methods and the class of $s$-stage 
Type II methods are distinct (though not disjoint).  So it is natural
to refer to a method as being of Type I or Type II, depending on 
which representation uses fewer stages; this convention is used
throughout the present work.
\end{rmk}
\begin{rmk}
Type I methods \eqref{tsrk_b} and Type II methods \eqref{tsrk_c} 
are equivalent to the (two-step) methods of Type 4 and Type 5, 
respectively, considered in \cite{constantinescu2009}.
\end{rmk}

%However, it does not 
%include the class of methods which includes the previous stage values $y_j^{n-l}$.

% {{\bf cbm}: I had a made a note to mention somewhere that the appropriate stage values cannot be simply recomputed if needed for the Jackiewicz methods to be SSP.  This is because they depend on $u^{n-2}$ which we no longer have access to.  Do we need to mention this?  Seems obvious now, but maybe that's because I left the note to myself!}

%=========================================================
\subsection{The Spijker Form for General Linear Methods}
%=========================================================
TSRK methods are a subclass of general linear methods.
In this section, we review the  theory of strong stability preservation for general 
linear methods  \cite{spijker2007}.   A general linear method can be written in the 
form
\begin{subequations}
  \label{spijkerform}
  \begin{align}
    \label{spijkerform-a}
    w_i^n &=  \sum_{j=1}^l s_{ij} x^{n}_j + \Dt \sum_{j=1}^m t_{ij} F(w_j^n), & (1\le i\le m), \\
    x^{n+1}_j &=  w^n_{J_j}, & (1\le j \le l).  \label{spijkerform-b}
  \end{align}
\end{subequations}
The terms $x^{n}_j $ are the  $l$  input values available from previous steps,
while the  $w_j^n$ includes both the output values and intermediate stages 
used to compute them.  Equation \eqref{spijkerform-b} indicates which of these
values are used as inputs in the next step.

We will frequently write the coefficients $s_{ij}$ and $t_{ij}$
as a $m\times l$ matrix  $\mS$ and a $m\times m$ matrix $\mT$,
respectively.
Without loss of generality (see \cite[Section 2.1.1]{spijker2007})
we assume that 
\be \label{eq:Ssum}
\mS \ve = \ve,
\ee
where $\ve$ is a vector with all entries equal to unity.
This implies that every stage is a consistent approximation to the solution
at some time.

Runge--Kutta methods, multi-step methods, and multi-step Runge--Kutta methods 
are all subclasses of general linear methods, and can be written
in the form \eqref{spijkerform}.  For example,
an $s$-stage Runge--Kutta method with Butcher coefficients $\mA$ and $\vb$
can be written in form \eqref{spijkerform} by taking $l=1, m=s+1$, $J=\{m\}$,
and 
\begin{align*}
\mS=\left(1 ,1 , \dots , 1 \right)\transpose,
& & \mT=\left(\begin{array}{cc} \mA & \mzero\\ \vb\transpose & 0 
\end{array}\right) .
\end{align*}
Linear multistep methods
\[ u^{n+1} = \sum_{j=1}^l \alpha_j u^{n+1-j} + \Dt  \sum_{j=0}^l \beta_j F \left(u^{n+1-j} \right), \]
admit the  Spijker form
\[ \mS= \begin{pmatrix} 
1      & 0      &  \dots & 0  \\
0      & \ddots &  \ddots& \vdots  \\
\vdots & \ddots &  \ddots& 0  \\
0      & \dots  &  0     & 1  \\
 \alpha_l & \alpha_{l-1} & \dots  & \alpha_1  \end{pmatrix}, \\ \; \; \; 
\mT^{(l+1)\times (l+1)} = \begin{pmatrix} 
0      & 0      &  \dots & 0  \\
0      & \ddots &  \ddots& \vdots  \\
\vdots & \ddots &  \ddots& 0  \\
0      & \dots  &  0     & 0  \\
 \beta_l & \beta_{l-1} & \dots  & \beta_0  \end{pmatrix}, 
\]
where $l$ is the number of steps, $m=l+1$, and $J=\{2, \dots, l+1\}$.
%$\mS$ is a $(l+1)\times l$  and $\mT$ is a $(l+1)\times (l+1)$ matrix, and $J=\{2, \dots, l+1\}$.

General TSRK methods \eqref{tsrk_JT} can be written in Spijker form 
as follows: set $m=2s+2$, $l=s+2$, $J=\{s+1,s+2,\dots,2s+2\}$, and
\begin{subequations} \label{tsrk_spijker}
\begin{align}
\bx^n & = \left(u^{n-1}, y_1^{n-1}, \dots, y_s^{n-1}, u^{n} \right)\transpose, \\
\bw^n & =\left(y_1^{n-1},y_2^{n-1}, \dots, y^{n-1}_s,u^{n},y_1^n,y_2^n, \dots, y^n_s, u^{n+1} \right)\transpose, \\
 \mS & = \begin{pmatrix}
\mzero & \mI & \mzero \\ 0 & \mzero & 1 \\ \vd & \mzero & \ve-\vd \\ \theta & \mzero & 1-\theta \end{pmatrix},\ \ \ \ \ \
\mT = \begin{pmatrix}
\mzero & \mzero & \mzero & \mzero \\
\mzero & 0 & \mzero & 0 \\
\mAh & \mzero & \mA   & \mzero \\
\vbh\transpose & \mzero & \vb\transpose & 0
\end{pmatrix}.
\end{align}
\end{subequations}

Type I methods \eqref{tsrk_b} can be written in a simpler form
with $m=s+2$, $l=2$, $J=\{1,  s+2\}$, and
\begin{align*}
\bx^n = \left(u^{n-1}, u^{n} \right)\transpose, \;  \; \; \; \; &
\bw^n =\left(u^{n},y_1^n,y_2^n, \dots, y^n_s, u^{n+1} \right)\transpose, \\
 \mS = \left( \begin{array}{cc}
0 & 1 \\ \vd & \ve-\vd \\ \theta & 1-\theta \end{array} \right),\ \ \ \ \ \
& \mT = \begin{pmatrix}
0 & \mzero & 0 \\
\mzero & \mA   & \mzero \\
0 & \vb\transpose & 0
\end{pmatrix}.
\end{align*}

Type II methods \eqref{tsrk_c} can also be written in a simpler form
with $m=s+2$, $l=2$, $J=\{2,  s+2\}$:
%$\vah$ denote the vector $[\hat{a}_1,\hat{a}_{2},\ldots,\hat{a}_{s}]\transpose$.  
\begin{align*}
\bx^n = \left(u^{n-1}, u^{n} \right)\transpose, \;  \; \; \; \; &
\bw^n =\left(u^{n-1},u^n,y_2^n, \dots, y^n_s, u^{n+1} \right)\transpose, \\
 \mS = \left( \begin{array}{cc}
1 & 0 \\ \vd & \ve-\vd \\ \theta & 1-\theta \end{array} \right),\ \ \ \ \ \
& \mT = \begin{pmatrix}
0 & \mzero & 0 \\
\vah & \mA   & \mzero \\
\hat{b}_1 & \vb\transpose & 0
\end{pmatrix}.
\end{align*}

%=========================================================
\subsection{The SSP Coefficient for General Linear Methods}
%=========================================================
In order to analyze the SSP property of a general linear method
\eqref{spijkerform}, we first define the vector 
$\vf=[F(w_1),F(w_2),\dots,F(w_m)]\transpose$, so that 
\eqref{spijkerform-a} can be written compactly as
\begin{align}
\label{spijkerform-compact}
\bw & = \mS\bx + \Dt \mT \vf.
\end{align}
Adding
%$r\mT \by$
$r\mT \bw$
to both sides of \eqref{spijkerform-compact} gives

\begin{align*}
    \left(\mI+r\mT\right) \bw & = \mS\bx 
                      + r\mT\left(\bw+\frac{\Dt}{r}\vf\right).
\end{align*}
Assuming that the matrix on the left is invertible  we  obtain,
\begin{align}  
    \bw & = (\mI+r\mT)^{-1} \mS\bx 
                      + r(\mI+r\mT)^{-1}\mT\left(\bw+\frac{\Dt}{r}\vf\right)  \nonumber \\ 
        & = \mR \bx + \mP\left(\bw+\frac{\Dt}{r}\vf\right), \label{canonical} 
\end{align}
where we have defined
\be   \label{canonical2}
\mP=r (\mI+r\mT)^{-1} \mT, \ \ \ \ \mR= (\mI+r\mT)^{-1} \mS= (\mI-\mP)\mS .
\ee

Observe that, by the consistency condition \eqref{eq:Ssum}, the row sums of $[\mR \ \mP]$ are 
each equal to one:
$$
\mR\ve+\mP\ve = (\mI-\mP)\mS\ve+\mP\ve = \ve - \mP\ve + \mP\ve = \ve.
$$
Thus, if $\mR$ and $\mP$ have no negative entries, each stage $w_i$ is
given by a convex combination of the inputs $x_j$ and the quantities
$w_j+(\Dt/r) F(w_j)$. 
In other words, this method is a convex combination of forward Euler steps.
Hence any strong stability property of the forward Euler method is 
preserved by the method \eqref{spijkerform-compact} under the time step
restriction given by $\dt \leq \sspcoef(\mS,\mT)\DtFE$ where
$\sspcoef(\mS,\mT)$ is defined as
\begin{align*}
\sspcoef(\mS,\mT) & = \sup_{r}\left\{r : (I+r \mT)^{-1} \mbox{ exists and }  \mP \ge 0, \mR \ge 0 
\right\},
\end{align*}
where $\mP$ and $\mR$ are defined in \eqref{canonical2}.
By the foregoing observation, it is clear that the SSP coefficient of method
\eqref{canonical} is greater than or equal to $\sspcoef(\mS,\mT)$.
%We shall see below that the two are equal unless the method is reducible.

To state precisely the conditions under which the SSP coefficient 
is, in fact, equal to $\sspcoef(\mS,\mT)$, we must introduce the concept of 
reducibility.  
%We would like to avoid writing a method in such a way that
%there exists a method with fewer stages that always 
%produces the same output. For this purpose, we introduce the concept of reducibility.
A Runge--Kutta method is said to be {\em reducible} if 
there exists a method with fewer stages that always produces the same output.
One kind of reducibility is known as {\em HS-reducibility};
a Runge--Kutta method is HS-reducible if two of its stages are identically
equal.  This definition is extended in a natural way to general linear
methods in \cite[Theorem~3.1]{spijker2007}; hereafter we refer to the reducibility
concept defined there also as {\em HS-reducibility}.
%The following definition of HS-reducibility for general linear methods
%comes from \cite[Remark~3.2]{spijker2007}.
%\begin{dfn}[HS-Reducibility]\label{dfn:hsr}
%A method in form \eqref{spijkerform} is {\em HS-reducible} if
%there exist indices $i,j, i\ne j$ such that all of the following hold: \\
%  \indent 1. Rows $i$ and $j$ of $\mS$ are equal. \\
%  \indent 2. Rows $i$ and $j$ of $\mT$ are equal.\\
%  \indent 3. Column $i$ of $\mT$ is not identically zero. \\
%  \indent 4. Column $j$ of $\mT$ is not identically zero. \\
%        Otherwise, we say the method is {\em HS-irreducible}.  
%\end{dfn}
%
%We note that this definition is more convenient than 
%\cite[Definition~2.6]{spijker2007} because we are interested in
%obtaining results that hold also for autonomous differential equations.
%The following theorem is essentially \cite[Theorem~3.1]{spijker2007}.

\begin{lem} (\cite[Theorem~3.1]{spijker2007}) \label{thm:spijker_coef}
Let $\mS,\mT$ be an HS-irreducible representation of a general linear
method.  Then the SSP coefficient of the method is 
$\sspcoef = \sspcoef(\mS,\mT)$.
\end{lem}

\subsection{Restrictions on the coefficients of SSP TSRK methods}
In light of Lemma~\ref{thm:spijker_coef}, we are interested
in methods with $\sspcoef(\mS,\mT)>0$.
The following lemma characterizes such methods.
\begin{lem}(\cite[Theorem 2.2(i)]{spijker2007})   \label{lem:inc}
$\sspcoef(\mS,\mT)>0$ if and only if all of the following hold:
\begin{subequations}   %\label{eq:lem:inc}
  \label{sspcc}
\begin{align}
\mS &\ge 0  \label{sspcc1}, \\
\mT &\ge 0 \label{sspcc2}, \\
\textup{Inc}(\mT\mS) &\le \textup{Inc}(\mS) \label{sspcc3}, \\
\textup{Inc}(\mT^2)  &\le \textup{Inc}(\mT) \label{sspcc4}.
\end{align}
\end{subequations}
where all the inequalities are element-wise and 
the incidence matrix of a matrix $\mM$ with entries $m_{ij}$ is 
\[
\textup{Inc}(\mM)_{ij} = \begin{cases}1 & \textrm{if } m_{ij} \ne 0 \\
                                    0 & \textrm{if } m_{ij} = 0.\end{cases}
\]
\end{lem}

To apply Lemma~\ref{thm:spijker_coef}, it is necessary to write
a TSRK method in HS-irreducible form.
A trivial type of HS-reducibility is the case where two
stages $y^n_i$, $y^n_j$ are identically equal; i.e., where the following 
condition holds for some $i\ne j$:
\be \label{stageequality}
d_i = d_j, \ \ \mbox{ rows $i,j$ of $\mA$ are identical, and rows $i,j$ of
$\mAh$ are identical}
\ee
This type of reducibility can be dealt with by simply combining the two stages;
hence in the following theorem we assume any such stages have been eliminated 
already.
Combining Lemma~\ref{thm:spijker_coef} and Lemma~\ref{lem:inc}, we
find that all SSP TSRK methods can be represented as Type I and Type II methods, 
introduced in Section~\ref{sec:intro}.

\begin{thm} \label{thm:allssptsrk}
Let $\mS,\mT$ be the coefficients of a $s$-stage TSRK method \eqref{tsrk_JT}
in the form \eqref{tsrk_spijker} with positive SSP coefficient $\sspcoef>0$
such that \eqref{stageequality} does not hold for any $i\ne j$.
Then the method can be written as an $s$-stage HS-irreducible method of Type 
I or Type II.
\end{thm} 
%Then \\
%(i) If one of the stages $y^n_j$ is identically equal to $u^n$, the method
%can be written as an HS-irreducible $s$-stage Type II method \eqref{tsrk_c}. \\
%(ii) Otherwise,
%the method can be written as an HS-irreducible $s$-stage Type I method \eqref{tsrk_b}. \\
%Hence all TSRK methods with positive SSP coefficient are
%of Type I \eqref{tsrk_b} or Type II \eqref{tsrk_c}.
\begin{proof}
Consider a method satisfying the stated assumptions.  Examination of $\mS,\mT$ 
reveals that the method is HS-irreducible in form \eqref{tsrk_spijker}
if there is no $y_j$ identically equal to $u^n$.
In this case, we can apply Lemma~\ref{thm:spijker_coef} 
to obtain that $\sspcoef(\mS,\mT)>0$.
Then condition \eqref{sspcc3} of Lemma~\ref{lem:inc} implies that
$\mAh=\vbh=\mzero$.
Under this restriction, methods of the form \eqref{tsrk_JT} simplify to
Type I methods \eqref{tsrk_b}.

Now consider the case in which $y_j=u^n$ for some $j$.
If necessary, reorder the stages so that $y_1^n=u^n$.  Then rows 
$s+1$ and $s+2$ of $[\mS \ \mT]$ in the
representation \eqref{tsrk_spijker} are equal.
Thus we can rewrite the method
in form (noting that also $y_1^{n-1}=u^{n-1}$)
as follows:
Set $m=2s+1$, $l=s+1$,  $J=\{s,s+1,s+2,\dots,2s+2\}$, and
\begin{subequations} \label{tsrk_spijker_reduced2}
\begin{align}
\bx^n & = \left(u^{n-1},y_2^{n-1}, \dots, y_s^{n-1}, u^{n} \right)\transpose, \\
\bw^n & =\left(u^{n-1},y_2^{n-1},y_3^{n-1}, \dots, y^{n-1}_s,y_1^n,y_2^n, \dots, y^n_s, u^{n+1} \right)\transpose, \\
 \mS & = \begin{pmatrix}
1 & \mzero & 0 \\ \mzero & \mI & \mzero \\  \vd & \mzero & \ve-\vd \\ \theta & \mzero & 1-\theta \end{pmatrix},\ \ \ \ \ \
\mT = \begin{pmatrix}
\mzero & \mzero & \mzero & \mzero \\
\mzero & 0 & \mzero & 0 \\
\mAh_1 & \mAh_{2:s} & \mA   & \mzero \\
\bh_1 & \vbh\transpose_{2:s} & \vb\transpose & 0
\end{pmatrix}.
\end{align}
\end{subequations}
Here $\mAh_1$ and $\bh_1$ represent the first column and first
element of $\mAh$ and $\vbh$, respectively, while
$\mAh_{2:s}$ and $\vbh\transpose_{2:s}$ represent the remaining
columns and remaining entries.  Since \eqref{tsrk_spijker_reduced2}
is HS-irreducible, we can apply Lemma~\ref{thm:spijker_coef}.
Then we have that $\sspcoef(\mS,\mT)>0$,
so that Lemma~\ref{lem:inc} applies.
Applying condition \eqref{sspcc3} of
Lemma~\ref{lem:inc} to the representation 
\eqref{tsrk_spijker_reduced2}, we find that $\mAh_{2:s}$ and 
$\vbh\transpose_{2:s}$ must vanish, but $\mAh_1$ and $\hat{b}_1$
may be non-zero.  The resulting methods are HS-irreducible Type II methods 
\eqref{tsrk_c}.
\end{proof}

\begin{rmk} \label{rmk2}
  % CBM: must be theorem not lemma, right?
  % One could formulate a version of Lemma~\ref{thm:spijker_coef}
  % dealing only with HS-irreducible methods of the form
  % \eqref{tsrk_JT}.
  One could formulate a version of Theorem~\ref{thm:allssptsrk} with
  the hypothesis that the method in form \eqref{tsrk_JT} is
  HS-irreducible.
  % i.e., ``Let $\mS,\mT$ be the coefficients of an
  % HS-irreducible $s$-stage TSRK method \eqref{tsrk_JT}''.
  This would lead only to the class of Type I methods, and in the
  resulting formalism the number of stages for Type II methods would
  be artificially increased by one (see Remark~\ref{rmk1}).  The
  advantage of explicitly considering the case $y^n_j=u^n$ is that we
  obtain Type II methods with the number of stages that accurately
  represents their cost.
  % Note that application of Lemma \ref{thm:spijker_coef} to
  % HS-irreducible methods of the form \eqref{tsrk_JT} leads only to
  % Type I methods, overlooking the class of Type II methods.  As
  % mentioned previously, this is equivalent to overestimating the
  % number of stages for Type II methods.
\end{rmk}

%\begin{thm} \label{thm:allssptsrk}
%All TSRK methods with positive SSP coefficient are
%of Type I \eqref{tsrk_b} or Type II \eqref{tsrk_c}.
%\end{thm} 
%\begin{proof}
%If the method is HS-irreducible in form \eqref{tsrk_spijker}, apply
%Corollary \ref{cor:typeI}.  If the method is reducible in this form, then
%either $y^n_i=y^n_j$ for some $i,j$, or $y^n_j=u^n$, for some $j$ (or both).
%In case $y^n_i=y_j^n$, one can simply remove stage 
%$y_i$ and adjust $\mA,\mAh$ correspondingly; this can be repeated as
%necessary until the method is irreducible or the reducibility is due
%only to the case $y^n_j=u^n$.  Then apply Corollary \ref{cor:typeII}.
%\end{proof}

%\subsection{Compact Notation for Type I \& II TSRK Methods}
We now introduce a compact, unified notation for Type I and Type II methods.  
This form is convenient for expressing the order conditions and restrictions
on the coefficients.  First we rewrite
an $s$-stage Type II method \eqref{tsrk_c} by including $u^{n-1}$ as one of 
the stages:
\begin{align*}
y_0^n &= u^{n-1}, \\
y_1^n &= u^{n}, \\
y_i^n &= d_i u^{n-1} + (1-d_i) u^n 
 + \Dt\sum_{j=0}^{s} a_{ij} F(y_j^n), &  2 \le i \le s,\\
u^{n+1} &= \theta u^{n-1} + (1-\theta) u^n 
 + \Dt\sum_{j=0}^{s} b_j F(y_j^n).
\end{align*}
Then both Type I and Type II methods can be written in 
the compact form
\begin{subequations} \label{tsrk_compact}
\begin{align}
\vy^n & = \bar{\vd} u^{n-1} + (\bone-\bar{\vd})u^n + \Dt \bar{\mA} \bff^n, \\
u^{n+1} & = \theta u^{n-1} + (1-\theta) u^n + \Dt \bar{\vb}\transpose \bff^n,
\end{align}
\end{subequations}
where, for Type I methods the coefficients with bars are equal to the 
corresponding coefficients without bars in \eqref{tsrk_b} and
\begin{align*}
\vy^n & = [y_1^n,\dots,y_s^n]\transpose, \ \ \bff^n = [F(y_1^n),\dots,F(y_s^n)]\transpose.
%\bar{\vd} & = [d_1,d_2,\dots,d_s]\transpose, \ \ \bar{\vb} = \vb, \ \ \bar{\mA}=\mA;
\end{align*}
Meanwhile, for Type II methods
\begin{align*}
\vy^n & = [u^{n-1},u^n,y_2^n,\dots,y_s^n]\transpose, \ \ \bff^n = [F(u^{n-1}),F(u^n),F(y_2^n),\dots,F(y_s^n)]\transpose, \\ 
\bar{\vd} & = [1,0,d_2,\dots,d_s]\transpose,
\ \ \bar{\vb} = [\hat{b}_1 \ \vb\transpose]\transpose, \ \ 
\bar{\mA}=\begin{pmatrix}0 & \mzero \\ \vah & \mA \end{pmatrix},
\end{align*}
where $d_j,\bb,\mA,\hat{b}_1,\vah$ refer to the coefficients in \eqref{tsrk_c}.

It is known that irreducible strong stability preserving Runge--Kutta methods
have positive stage coefficients, $a_{ij}\ge0$ and strictly positive weights,
$b_j>0$.  The following theorem shows that similar properties hold for SSP TSRK
methods.  The theorem and its proof are very similar to 
\cite[Theorem~4.2]{kraaijevanger1991}.
In the proof, we will use a
second irreducibility concept.  A method is said to be {\em DJ-reducible}
if it involves one or more stages whose value does not affect the output.
If a TSRK method is neither HS-reducible nor DJ-reducible, we say it is
{\em irreducible}.

\begin{thm} \label{thm:pos_coef}
The coefficients of an HS-irreducible TSRK method of Type I \eqref{tsrk_b} or 
Type II \eqref{tsrk_c}
with positive SSP coefficient satisfy the following bounds:
\begin{align} \label{condition_i}
\bar{\mA}\ge0,\bar{\vb}\ge0, 0\le\bar{\vd}\le1, \textup{ and } 0\le\theta\le1.
\end{align}
Furthermore, if the method is also DJ-irreducible, the weights must be
strictly positive: 
\begin{align} \label{condition_ii}\bar{\vb}>0.\end{align}
All of these inequalities should be interpreted component-wise.
\end{thm}
\begin{proof}
%By Theorem \ref{thm:allssptsrk}, any HS-irreducible method can be written
%in either form \eqref{tsrk_b} or form \eqref{tsrk_c}.
Application of Lemma \ref{thm:spijker_coef} implies that $\sspcoef(\mS,\mT)>0$.
Therefore Lemma~\ref{lem:inc} applies.
The first result \eqref{condition_i} then follows from conditions \eqref{sspcc1}
and \eqref{sspcc2} of Lemma~\ref{lem:inc}.
%  This proves the first part.
%Suppose $\sspcoef(\mS,\mT)>0$.  Then by \cite[Theorem 2.2(ii)]{spijker2007}),
%for $r\in[0,\sspcoef(\mS,\mT)]$, $(\mI+r\mT)^{-1}$ exists and
%\begin{align*}
%(\mI+r\mT)^{-1}\mT & \ge 0 \\
%(\mI+r\mT)^{-1}\mS & \ge 0
%\end{align*}
%In particular, taking $r=0$, we have that $\mS,\mT\ge0$, which
%immediately implies non-negativity of $\mA,\vb,\vd,\theta,\ve-\vd$,
%and $1-\theta$.
%To prove the necessity of condition (ii),
%suppose $\sspcoef(\mS,\mT)>0$.  Then by \cite[Theorem 2.2(ii)]{spijker2007}),
%for $r\in[0,\sspcoef(\mS,\mT)]$, $(\mI+r\mT)^{-1}$ exists and
%\begin{align*}
%(\mI+r\mT)^{-1}\mT & = \mT - r\mT^2 + \cdots \ge 0.
%\end{align*}
%Thus if any entry of $\mT$ vanishes, the corresponding entry
%of $\mT^2$ must also vanish.  
To prove the second part, observe that condition \eqref{sspcc4} of
Lemma~\ref{lem:inc} means that if $b_j=0$ for some $j$ then
\be \label{eq:beqz} \sum_i b_i a_{ij} = 0.\ee
Since $\mA,\vb$ are non-negative, \eqref{eq:beqz} implies that 
either $b_i$ or $a_{ij}$ is zero for each value of $i$.
Now partition the set $\Sop=\{1,2,\dots,s\}$ into $\Sop_1,\Sop_2$
such that $b_j>0$ for all $j\in\Sop_1$ and $b_j=0$ for all $j\in\Sop_2$.
Then $a_{ij}=0$ for all $i\in\Sop_1$ and $j\in\Sop_2$.
This implies that the method is DJ-reducible, unless $\Sop_2=\emptyset$.
\end{proof}

%From \eqref{sspcc4} we have   
%\begin{align}
%\Inc(\mA^2)& \le & \Inc(\mA) \\
%\end{align}

%From  \eqref{sspcc3}, we have
%\begin{align}
%\Inc(\mA (\ve-\vd) )& \le & \Inc(\ve-\vd) \\
%\Inc(b\transpose (\ve-\vd) )& \le & \Inc(\ve-\vd) \\
%\end{align}
%Additionally, for Type I, we have
%\begin{align}
%\Inc(\mA \vd) )& \le & \Inc(\vd) \\
%\Inc(\vb\transpose \vd) )& \le & \Inc(\theta)
%\end{align}

%% file: tsrk_ocs.tex
Order conditions for TSRK methods up to order 6 have previously been derived in
\cite{jackiewicz1995}.  However, two of the sixth-order conditions therein
appear to contain errors (they do not make sense dimensionally).
Alternative approaches to order conditions for TSRK methods, using trees
and B-series, have also been identified \cite{butcher1997,hairer1997}.

%{\bf We should mention B-series and cite the other papers on this.
%  cbm: I think Jim Verner and Anne Kaverno (spl?) have a new paper on this.
%  dk: I couldn't find such a paper, but cited two old ones above.}

In this section we derive order conditions for TSRK methods of Types I
and II.  The order conditions derived here are not valid for 
for the general class of methods given by \eqref{tsrk_JT}.
Our derivation follows Albrecht's approach \cite{albrecht1996},
and leads to very simple conditions, which are almost identical in
appearance to 
order conditions for RK methods.  For simplicity of notation, we 
consider a scalar ODE only.
For more details and justification of this approach for systems,
see \cite{albrecht1996}.

\subsection{Derivation of Order Conditions}
When applied to the trivial ODE $u'(t)=0$, any TSRK scheme reduces to
the recurrence $u^{n+1} = \theta u^{n-1} + (1-\theta)u^{n}$.  For a
SSP TSRK scheme, we have $0\le \theta \le 1$ (by
Theorem~\ref{thm:pos_coef}) and it follows that the method is
zero-stable. 
% cbm, added "it follows" to indicate you have to do *some* work here,
% showing that the roots of the
% recurrence relation lie in/on the unit disc.
Hence to prove convergence of order $p$, it is sufficient to prove
consistency of order $p$ (see, e.g., \cite[Theorem~2.3.4]{jackiewicz2009}).

Let $\tu(t)$ denote the exact solution at time $t$ and define
\begin{align*} \bty^n=[\tu(t_n+c_1\Dt),\dots,\tu(t_n+c_s\Dt)] \\
\btf^n=[F(\tu(t_n+c_1\Dt)),\dots,F(\tu(t_n+c_s\Dt))],
\end{align*}
where $\v{c}$ identifies the abscissae of the TSRK scheme.  These
represent the correct stage values and the corresponding correct
values of $F$.  Then the \emph{truncation error} $\lte^n$ and
\emph{stage truncation errors} $\ste^n$ are implicitly defined by
\begin{subequations}\label{eq:tsrk_lte}
\begin{align}
\bty^n & = \bar{\vd} \tu^{n-1} + (\ve-\bar{\vd})\tu^n + \Dt \mAb \btf^n + \Dt \ste^n, \\
\tu(t_{n+1}) & = \theta\tu^{n-1} + (1-\theta)\tu^{n} + \Dt \bbb\transpose\btf^n + \Dt \lte^n.
\end{align}
\end{subequations}

To find formulas for the truncation errors, we make use of
the Taylor expansions
\begin{align*}
\tu(t_{n}+c_i\Dt)
& = \sum_{k=0}^\infty \frac{1}{k!} \Dt^k c_i^k \tu^{(k)}(t_{n}), \\
F(\tu(t_{n}+c_i\Dt)) = \tu'(t_{n}+c_i\Dt)
& = \sum_{k=1}^\infty \frac{1}{(k-1)!} \Dt^{k-1} c_i^{k-1} \tu^{(k)}(t_{n}), \\
\tu(t_{n-1})
& = \sum_{k=0}^\infty \frac{(-\Dt)^k}{k!} \tu^{(k)}(t_{n}).
\end{align*}
Substitution gives
\begin{subequations} \label{eq:tsrk_ste}
\begin{align}
\ste^n & = \sum_{k=1}^\infty \ste_k \Dt^{k-1} \tu^{(k)}(t_{n}), \\
\lte^n & = \sum_{k=1}^\infty \lte_k \Dt^{k-1} \tu^{(k)}(t_{n}),
\end{align}
\end{subequations}
where
\begin{align*} 
\ste_k & = \frac{1}{k!} \left(\bc^k-(-1)^k\bar{\vd}\right) - \frac{1}{(k-1)!} \mAb \bc^{k-1}, \\
\lte_k & = \frac{1}{k!}\left(1-(-1)^k\theta\right) - \frac{1}{(k-1)!} \bbb\transpose \bc^{k-1}.
\end{align*}
Subtracting \eqref{eq:tsrk_lte} from \eqref{tsrk_compact} gives
\begin{subequations} \label{eq:tsrk_gserr_series}
\begin{align}
\gserr^n & = \bar{\vd} \gerr^{n-1} + (\ve-\bar{\vd})\gerr^{n} + \Dt \mAb \rhsserr^n - \Dt \ste^n, \\
\gerr^{n+1} & = \theta \gerr^{n-1} + (1-\theta)\gerr^{n} + \Dt \bbb\transpose \rhsserr^n - \Dt \lte^n,
\end{align}
\end{subequations}
where $\gerr^{n+1}=u^{n+1}-\tu(t_{n+1})$ is the global error,
$\gserr^n = \by^n-\tilde{\by^n}$, is the global stage error, and
$\rhsserr^n = \bff^n-\btf^n$ is the right-hand-side stage error.

%Our objective is to determine conditions on $\mAb,\bbb,\bar{\vd},\theta$ such that
%$\gerr^{n+1}=\Oop(\Dt^p)$ whenever $\gerr^{n}=\Oop(\Dt^p)$ and 
%$\gerr^{n-1}=\Oop(\Dt^p)$.
If we assume an expansion for the right-hand-side stage errors
$\rhsserr^n$ as a power series in $\Dt$
\begin{align} \label{eq:rhserr_series}
\rhsserr^n & = \sum_{k=0}^{p-1} \rhsserr_k^n \Dt^{k} + \Oop(\Dt^{p}), %\\
%\gserr^n & = \sum_{k=0}^{p-1} \gserr_k^n \Dt^{k}+ \Oop(\Dt^{p}),
\end{align}
then 
%assuming $\gerr^{n}=\Oop(\Dt^p)$ and
substituting the expansions \eqref{eq:rhserr_series} and \eqref{eq:tsrk_ste}
into the global error formula \eqref{eq:tsrk_gserr_series} yields
\begin{subequations} \label{eq:gserr_rec_both}
\begin{align} \label{eq:gserr_rec}
\gserr^n & = \bar{\vd} \gerr^{n-1} + (\ve - \bar{\vd})\gerr^n + \sum_{k=0}^{p-1} \mAb \rhsserr^n_k \Dt^{k+1} -\sum_{k=1}^{p} \ste_k \tu^{(k)}(t_{n}) \Dt^k + \Oop(\Dt^{p+1}), \\
\gerr^{n+1} & = \theta \gerr^{n-1} + (1-\theta)\gerr^n + \sum_{k=0}^{p-1} \bbb\transpose\rhsserr^n_{k} \Dt^{k+1} -\sum_{k=1}^{p} \lte_k \tu^{(k)}(t_{n}) \Dt^k + \Oop(\Dt^{p+1}).
\end{align}
\end{subequations}
Hence
%combining this with the stability of the recurrence relation $\gerr^{n+1} = \theta \gerr^{n-1} + (1-\theta)\gerr^n$ for $0 \le \theta \le 1$, 
we find the method is consistent of order $p$ if
%Hence the method is consistent of order $p$ if
\be \label{eq:ocs}
\bbb\transpose\rhsserr_k^n=0 \ \ (0\le k\le p-1) \textup{ \qquad and \qquad } \lte_k=0 \ \ (1\le k\le p).
%\bbb\transpose\rhsserr_k^n=0 \textup{ and } \lte_k=0 \textup{ for } 0\le k\le p-1.
\ee
%\textbf{Colin: changed second of these to go from 1 to $p$ (was 0
%  to $p$), and edited the text...  I think we need stability of that
%  recurrence relation to have consistency of our TSRK.  Then
%  zero-stability of our TSRK + consistency of our TSRK implies
%  convergence of our TSRK.  (Zero-stability of the TSRK amounts to the
%  same condition as stability of the recurrence relation.)  The
%  recurrence relation is only neutrally stable for $\theta=1$, but
%  maybe we don't need to get into that...}

It remains to determine the vectors 
$\rhsserr_k^n$ in the expansion \eqref{eq:rhserr_series}.  In fact, 
we can relate these recursively to the $\gserr_k$.  First we define
\begin{align*}
\bt_n & = t_n\ve+\bc\Dt, \\
\bFf(\by,\bt) & = [F(y_1(t_1)),\dots,F(y_s(t_s))]\transpose.
\end{align*}
Then we have the Taylor series
\begin{align*}
  \bff^n =  \bFf(\by^n,\bt_n) & = \btf^n + \sum_{j=1}^\infty \frac{1}{j!} (\by^n-\bty^n)^j \cdot
        \bFf^{(j)}(\bty^n,\bt_n) \\
       & = \btf^n + \sum_{j=1}^\infty \frac{1}{j!} (\gserr^n)^j \cdot \bg_j(\bt_n),
\end{align*}
where 
\begin{align*}
\bFf^{(j)}(\by,\bt) = [F^{(j)}(y_1(t_1)),\dots,F^{(j)}(y_s(t_s))]\transpose, \\
\bg_j(\bt) = [F^{(j)}(y(t_1)),\dots,F^{(j)}(y(t_s))]\transpose,
\end{align*}
and the dot product denotes component-wise multiplication. Thus
\begin{align*}
  \rhsserr^n = \bff^n - \btf^n = \sum_{j=1}^\infty \frac{1}{j!} (\gserr^n)^j \cdot
        \bg_j(t_n\ve+\bc\Dt).
\end{align*}
Since
\begin{align*}
\bg_j(t_n\ve+\bc) & = \sum_{l=0}^\infty \frac{\Dt^l}{l!} \mC^l \bg_j^{(l)}(t_{n}),
\end{align*}
where $\mC=\textup{diag}(\bc)$, we finally obtain the desired expansion:
\begin{align} \label{eq:triplesum}
\rhsserr^n & = \sum_{j=1}^\infty \sum_{l=0}^\infty \frac{\Dt^l}{j!l!} \mC^l (\gserr^n)^j \cdot \bg_j^{(l)}(t_{n}).
\end{align}
To determine the coefficients $\rhsserr_k$, we alternate recursively 
between \eqref{eq:triplesum} and \eqref{eq:gserr_rec}.
Typically, the abscissae $\vc$ are chosen as $\vc=\mAb\ve$ so that $\ste_1=0$.
With these choices, we collect the terms relevant for up to fifth-order accuracy: \\
\begin{tabular}{|cl|} \hline
 Terms appearing in  $\rhsserr_1$: &$\emptyset$\\
Terms appearing in  $\gserr_2$: & $\ste_2$ \\
Terms appearing in   $\rhsserr_2$: & $\ste_2$ \\
Terms appearing in  $\gserr_3$: & $\mAb\ste_2,\ste_3$ \\
Terms appearing in  $\rhsserr_3$: &$\mC\ste_2, \mAb\ste_2, \ste_3$ \\
Terms appearing in  $\gserr_4$:  & $\mAb \mC\ste_2, \mAb^2\ste_2, \mAb\ste_3,\ste_4$ \\
Terms appearing in  $\rhsserr_4$:  & $\mAb \mC\ste_2, \mAb^2\ste_2, \mAb\ste_3,\ste_4, \mC\mAb\ste_2, \mC\ste_3, \mC^2\ste_2, \ste_2^2$ \\ \hline
 \end{tabular}
 
The order conditions are then given by \eqref{eq:ocs}.
In fact, we are left with order conditions identical to those
for Runge--Kutta methods, except that the definitions of the stage truncation
errors $\ste_k, \lte_k,$ and of the abscissas $\bc$ are modified.
For a list of the order conditions up to eighth order, see 
\cite[Appendix~A]{albrecht1987}.

\subsection{Order and Stage Order of TSRK Methods}
The presence of the term $\ste_2^2$ in $\rhsserr_4$ leads to the
order condition $\vb\transpose\ste_2^2=0$.  For SSP methods, since
$\vb>0$ (by Theorem \ref{thm:pos_coef}), this implies that $\ste_2^2=0$, 
i.e. fifth-order SSP TSRK methods
must have stage order of at least two.  The corresponding condition for 
Runge--Kutta methods leads to the well-known result that no explicit RK
method can have order greater than four and a positive SSP coefficient.  
Similarly, the 
conditions for seventh order will include $\vb\transpose\ste_3^2=0$, which
leads (together with the non-negativity of $\mA$) to the result that
implicit SSP RK methods have order at most six.  In general, the
conditions for order $2k+1$ will include the condition
$\vb\transpose\ste_k^2=0$.
Thus, like SSP Runge-Kutta methods, SSP TSRK methods
have a lower bound on the stage order, and an upper bound on
the overall order.

\begin{thm} \label{thm:min_so}
Any TSRK method \eqref{tsrk_JT} of order $p$ with positive
SSP coefficient has stage order at least $\lfloor\frac{p-1}{2}\rfloor.$
\end{thm}
\begin{proof}
Without loss of generality, we consider only irreducible methods; thus
we can apply Theorems \ref{thm:allssptsrk} and \ref{thm:pos_coef}.
Following the procedure outlined above, we find that for order $p$, the
coefficients must satisfy
$$ \vb\transpose\ste_k^2 = 0, \quad \quad k=1,2,\dots,\left\lfloor\frac{p-1}{2}\right\rfloor.$$
Since $\vb>0$ by Theorem~\ref{thm:pos_coef}, this implies that
$$ \ste_k^2 = 0, \quad \quad k=1,2,\dots,\left\lfloor\frac{p-1}{2}\right\rfloor.$$
\end{proof}

Application of Theorem~\ref{thm:min_so} dramatically simplifies the
order conditions for high order SSP TSRKs.  This is because increased
stage order leads to the vanishing of many of the order conditions.
Additionally, Theorem \ref{thm:min_so} leads to an upper bound on the
order of explicit SSP TSRKs.

\begin{thm}\label{thm:order_barrier}
The order of an explicit SSP TSRK method is at most eight.  Furthermore,
if the method has order greater than six, it is of Type II.  
\end{thm}
\begin{proof}
Without loss of generality, we consider only irreducible methods; thus
we can apply Theorems \ref{thm:pos_coef} and \ref{thm:min_so}.

To prove the second part, consider an explicit irreducible TSRK method with 
order greater
than six.  By Theorem~\ref{thm:min_so}, this method must have stage order
at least three.  Solving the conditions for stage $y_2$ to have stage order three
gives that $c_2$ must be equal to $-1$ or $0$.  
Taking $c_2=-1$ implies that $y_2=u^{n-1}$, so the method is of type II.
Taking $c_2=0$ implies $y_2=y_1=u^n$; in this case, there must be some
stage $y_j$ not equal to $u^n$ and we find that necessarily $c_j=-1$ and
hence $y_j=u^{n-1}$.

To prove the first part,
suppose there exists an irreducible SSP TSRK method \eqref{tsrk_compact} of 
order nine.
By Theorem \ref{thm:min_so}, this method must have stage order at least four.
Let $j$ be the index of the first stage that is not identically equal to
$u^{n-1}$ or $u^n$.
Solving the conditions for stage $y_j$ to have order four reveals that
$c_j$ must be equal to $-1,0$, or $1$.  The cases $c_j=-1$ and $c_j=0$
lead to $y_j=u^{n-1}$ and $y_j=u^n$, contradicting our assumption.
Taking $c_j=1$ leads to $d_j=5$.
By Theorem~\ref{thm:pos_coef}, this implies that the method is not SSP.
\end{proof}

We remark here that other results on the structure of SSP
TSRK methods may be obtained by similar use of the stage order conditions
and Theorems \ref{thm:pos_coef} and \ref{thm:min_so}.  We list some examples 
here, but omit the
proofs since these results are not essential to our present purpose.

\begin{enumerate}
  \item Any SSP TSRK method (implicit or explicit) of order greater than
        four must have a stage equal to $u^{n-1}$ or $u^n$.
  \item The abscissae $c_i$ of any SSP TSRK method of order greater than
        four must each be non-negative or equal to $-1$.
  \item (Implicit) SSP TSRK methods with $p>8$ must be of Type II.
\end{enumerate}

%% file: twostep.tex
Our objective in this section is to find SSP TSRK methods that 
have the largest possible SSP coefficient.
A method of order $p$ with $s$ stages is said to be optimal if it has the largest
value of $\sspcoef$ over all TSRK methods with order at least $p$ with
no more than $s$ stages.

The methods presented were found via numerical search using \Matlab's
Optimization and Global Optimization toolboxes.  We searched over Type
I and Type II methods; the optimal methods found are of Type II in
every case.  For the methods of seventh and eighth order, this is known
{\em a priori} from Theorem \ref{thm:order_barrier}.  
Even for the lower order methods, 
this is not surprising, since explicit $s$-stage Type II methods
\eqref{tsrk_c} have an additional $s-1$ degrees of freedom compared to
explicit $s$-stage Type I methods \eqref{tsrk_b}.
Although we do not know in general if these methods are
globally optimal, our search recovered the global optimum in every
case for which it was already known.

\subsection{Formulating the Optimization Problem}
\label{sec:formulate_opt_prob}
The optimization problem is formulated using the theory of Section \ref{sect:sspglms}:
  \begin{gather*}   % can also use align here instead of gather
    \max_{\mS, \mT} r,   \\
    \text{subject to} \quad \left\{
      \begin{aligned}
        (\mI+r\mT)^{-1}  \mS &\ge  0, \\
            (\mI+r\mT)^{-1}  \mT &\ge  0, \\
        \Phi_p(\mS,\mT)  &=  0,
      \end{aligned}\right.
  \end{gather*}
where the inequalities are understood component-wise and 
$\Phi_p(\mS,\mT)$ represents the order conditions up to order $p$.  This
formulation, solved numerically in \Matlab using a sequential quadratic
programming approach (\code{fmincon} in the optimization toolbox), was
used to find the methods given below.

In comparing methods with different numbers of stages, one is usually
interested in the relative time advancement per computational cost.
For this purpose, we define the {\em effective SSP coefficient}
\[ \ceff(\mS,\mT) = \frac{\sspcoef(\mS,\mT)}{s}.\]
This normalization enables us to compare the
cost of integration up to a given time, assuming that the time
step is chosen according to the SSP restriction.
%{{\bf cbm:} is this a good place to comment/remind that this concept of stages is a bit complicated for GLM?}

%Type II methods \eqref{tsrk_c} have one degree of freedom less than Type
%I methods \eqref{tsrk_b} with the same number of stages.
%However, if the methods are required to be explicit (i.e., $\mA$ is
%strictly lower triangular), then Type II methods \eqref{tsrk_c} have $s-1$
%degrees of freedom {\em more} than Type I methods \eqref{tsrk_b}.  
%Hence we might expect to find better explicit methods among the former 
%class and better implicit methods among the latter.  We shall see that
%this turns out to be the case.

\begin{rmk}
By virtue of Theorem \ref{thm:allssptsrk}, optimal SSP methods found in the
classes of Type I and Type II TSRK methods are in fact optimal over
the larger class of methods \eqref{tsrk_JT}.  Also, because they do
not use intermediate stages from previous timesteps, special
conditions on the starting method (important for methods of the form
\eqref{tsrk_JT} \cite{hairer1997,verner2006,verner2006b}) are
unnecessary.  Instead, the method can be started with any SSP Runge--Kutta
method of the appropriate order.
\end{rmk}

\begin{rmk}
The optimal Type II methods found here could be rewritten as Type I
methods by adding a stage (see Remark \ref{rmk1}).  However, this
would reduce their effective SSP coefficient and render them (apparently)
less efficient than some other (Type I) methods.  This indicates
once more the importance of explicitly accounting for the practical
difference between Type I and Type II methods.
\end{rmk}

\subsection{Low-storage implementation of Type II SSP TSRKs}
The form \eqref{canonical}, with $r= \sspcoef(\mS,\mT)$, typically 
yields very sparse coefficient matrices for optimal Type II SSP TSRK methods.
This form is useful for a low-storage implementation.
Written out explicitly, this form is:
\begin{subequations}
\label{tsrk_so}
\begin{align}
y_i^n &=  \tilde{d}_i u^{n-1}+\left(1-\tilde{d}_i-\sum_{j=0}^{s} q_{ij}\right)  u^n 
 + \sum_{j=0}^{s} q_{ij} \left( y^n_j + \frac{\Dt}{r} F(y_j^n) \right), \ \ (1\le i \le s), \\
u^{n+1} &=  \tilde{\theta} u^{n-1} 
 + \left(1-\tilde{\theta}-\sum_{j=0}^{s} \eta_{j}\right)u^n 
 + \sum_{j=0}^{s} \eta_{j} \left( y^n_j
 +\frac{\Dt}{r} F(y^n_j) \right),
\end{align}
\end{subequations}
where the coefficients are given by (using the relations \eqref{canonical2}):
\begin{align*}
 \mQ= r\mAb(\mI+r\mAb)^{-1}, && \mateta= r\bbb\transpose(\mI+r\mAb)^{-1}, \\
\tilde{\vd} = \bar{\vd}-\mQ \bar{\vd}, & & \tilde{\theta} = \theta-\mateta\transpose \bar{\vd}.
\end{align*}

When implemented in this form, many of the methods presented in the next
section have modest storage requirements, despite using large numbers of
stages.  The analogous form for Runge--Kutta methods was used in 
\cite{ketcheson2008}.

In the following sections we discuss the numerically optimal methods,
and in Tables~\ref{tbl-tsrk125}, \ref{tbl-tsrk126},
\ref{tbl-tsrk127}, and~\ref{tbl-tsrk128}, we give the coefficients in
the low-storage form \eqref{tsrk_so} for some numerically optimal
methods.

\subsection{Optimal Methods of Orders One to Four}
%  This approach recovered the methods of \cite{} up to fourth order and four stages,
%which were found using a software package that provides a numerical certificate of global optimality.
%Because our approach was able to find these
%previously known methods, we expect that some of new
%methods---particularly those of lower-order or lower number of
%stages---may be globally optimal.

In the case of first-order methods, one can do no better
(in terms of effective SSP coefficient) than the forward Euler method.
For orders two to four, SSP coefficients of optimal methods found by
numerical search are listed in
Table~\ref{tbl-2step}.  We list these mainly for completeness, since
SSP Runge--Kutta methods with good properties exist up to order four.

In \cite{ketcheson2009a}, upper bounds for the values in Table
\ref{tbl-2step} are found by computing optimally contractive general
linear methods for linear systems of ODEs.  Comparing the present
results to the two-step results from that work, we see that this
upper bound is achieved (as expected) for all first and second order
methods, and even for the two- and three-stage third-order methods.

Optimal methods found in \cite{constantinescu2009} include two-step
general linear methods of up to fourth order using up to four stages.
By comparing Table~\ref{tbl-2step} with the results therein, we see that
the SSP coefficients of the optimal methods among the classes examined
in both works (namely, for $1\le s\le 4$, $2\le p\le 4$) agree.
The methods found in \cite{constantinescu2009} are produced by software
that guarantees global optimality.
%Furthermore, in some cases (such as the four-stage, fourth order
%method) we find a representation of the optimal method that requires
%less storage than that found in \cite{constantinescu2009}.

All results listed in bold are thus known to be optimal because
they match those obtained in \cite{constantinescu2009},
\cite{ketcheson2009}, or both.  This demonstrates that
our numerical optimization approach was able to recover
all known globally optimal methods, and
suggests that the remaining methods found in the present work
may also be globally optimal.

The optimal $s$-stage, second-order SSP TSRK method is in fact both
a Type I and Type II method, and was found in numerical searches over
methods of both forms.
It has SSP coefficient $\sspcoef=\sqrt{s(s-1)}$ and nonzero coefficients
\begin{align*}
q_{i,i-1} &=  1, & (2\le i\le s), \\
q_{s+1,s} &=  2(\sspcoef-s+1),  \\
\tilde{\vd} &= \mzero, \\
 \tilde{\theta} &=  2(s-\sspcoef)-1.
\end{align*}
Note that these methods have $\ceff=\sqrt{\frac{s-1}{s}}$, whereas the
corresponding optimal Runge--Kutta methods have $\ceff=\frac{s-1}{s}$.
%For $s=2, 3, 4$, this is a significant improvement in $\ceff$: 41\%,
%22\%, 15\% larger, respectively, than the optimal explicit
%Runge--Kutta method with the same number of stages.  
Using the low-storage assumption introduced in \cite{ketcheson2008}, these
methods can be implemented with just three storage registers, just one register
more than is required for the optimal second-order SSP Runge--Kutta methods.

The optimal nine-stage, third-order method is remarkable in that it
is a Runge--Kutta method.  In other words, allowing the freedom of using
an additional step does not improve the SSP coefficient in this case.

\begin{table}
\caption{Effective SSP coefficients $\ceff$ of optimal
  explicit 2-step Runge--Kutta methods of order two to four.
  Results known to be optimal from \cite{constantinescu2009} or 
  \cite{ketcheson2009} are shown in \textbf{bold}.
\label{tbl-2step}}
\center
\begin{tabular}{l|ccc} \hline
$s$ $\backslash$ $p$ & 2     & 3     & 4     \\ \hline
2                &\bf 0.707& \bf 0.366 & \\
3                &\bf 0.816& \bf 0.550 & \bf 0.286 \\
4                &\bf 0.866& \bf 0.578 & \bf 0.398 \\
5                &\bf 0.894 & 0.598 & 0.472 \\
6                &\bf 0.913 & 0.630 & 0.509 \\
7                &\bf 0.926 & 0.641 & 0.534 \\
8                &\bf 0.935 & 0.653 & 0.562 \\
9                &\bf 0.943 & 0.667 & 0.586 \\
10               &\bf 0.949 & 0.683 & 0.610 \\ \hline
\end{tabular}
\end{table}

\subsection{Optimal Methods of Orders Five to Eight}
Table \ref{tbl-2step_58} lists effective SSP coefficients of numerically
optimal TSRK methods of orders five to eight.  Although these methods
require many stages, it should be remembered that high-order (non-SSP) Runge--Kutta
methods also require many stages.  Indeed, some of our SSP TSRK
methods have fewer stages than the minimum number required to achieve
the corresponding order for an Runge--Kutta method (regardless of SSP
considerations).

The fifth-order methods present an unusual phenomenon: when the number
of stages is allowed to be greater than eight, it is not 
possible to achieve a larger effective SSP coefficient than the
optimal 8-stage method, even allowing as many as twelve stages.
This appears to be accurate, and not simply due to failure of
the numerical optimizer, since in the nine-stage case the optimization
scheme recovers the apparently optimal method in less than one minute,
but fails to find a better result after several hours.

The only existing SSP methods of order greater than four are the
hybrid methods of Huang \cite{huang2009}.  
% \textbf{TODO: we'll probably need citations to Vaillancourt's papers
% here once they start to appear}
Comparing the best
TSRK methods of each order with the best hybrid methods of each order,
the TSRK methods have substantially larger effective SSP coefficients.

The effective SSP coefficient is a fair metric for comparison
between methods of the same order of accuracy.  
Furthermore, our twelve-stage TSRK methods have sparse
coefficient matrices and can be implemented in the low-storage form \eqref{tsrk_so}.
Specifically, the fifth- through eighth-order methods of twelve stages require
only 5, 7, 7, and 10 memory locations per unknown, respectively,
under the low-storage assumption employed in 
\cite{ketcheson2008,ketcheson2010}.  Typically the methods with fewer
stages require the same or more storage, so
there is no reason to prefer methods with fewer stages if they have
lower effective SSP coefficients.  Thus, for sixth through eighth
order, the twelve-stage methods seem preferable.
The SSP TSRK methods recommended here even require less storage than what
(non-SSP one-step) Runge--Kutta methods of the corresponding
order would typically use.

In the case of fifth order methods, the eight-stage method 
has a larger effective SSP coefficient than the twelve stage method,
so the eight stage method seems best in terms of efficiency.  
However the eight stage method 
requires more storage registers (6) than the twelve
stage method (5). So while  the eight stage method  might be preferred
for efficiency, the twelve stage method is preferred for low storage
considerations.

%We have found methods with larger SSP coefficients by allowing
%even more stages.  However, the cost of rejecting a step in a
%method with so many stages seems to become a significant drawback.

\begin{table}
\caption{Effective SSP coefficients $\ceff$ of optimal
  explicit two-step Runge--Kutta methods of order five to eight.
  %For comparison,
  %we list also the largest SSP coefficient for each order
  %among all methods of the following classes:
  %HM = Hybrid methods;
  %DRK = Downwind Runge--Kutta Methods;
  %LMM = Linear multistep methods.
\label{tbl-2step_58}}
\center
\begin{tabular}{l|cccc} \hline
$s$ $\backslash$ $p$ & 5      & 6    & 7    & 8 \\ \hline
4                & 0.214 & & & \\
5                & 0.324 & & & \\ 
6                & 0.385 & 0.099 & & \\
7                & 0.418 & 0.182 & & \\ 
8                & 0.447 & 0.242 & 0.071 & \\ 
9                & 0.438 & 0.287 & 0.124 & \\
10               & 0.425 & 0.320 & 0.179 & \\ 
11               & 0.431 & 0.338 & 0.218 & 0.031 \\
12               & 0.439 & 0.365 & 0.231 & 0.078 \\ \hline
%DWRK             & 0.340 &       &       \\ \hline
%LMM              & 0.411 \\ \hline
%HM               & 0.373 & 0.220 & 0.117 \\ \hline
\end{tabular}
\end{table}

%% file: numerics.tex
\subsection{Start-up procedure}
\label{sec:startup}

As mentioned in Section~\ref{sec:formulate_opt_prob}, TSRK are not
self-starting and thus require startup procedures, and while in
general this is somewhat complicated, for our class of methods it is
straightforward.
%  We emphasize that the startup procedure
%is not critical for this class of TSRK methods, essentially because
%they do depend on previous stage values.
We only require that the starting procedure be of sufficient accuracy
and that it also be strong stability preserving.

Figure~\ref{fig:startup} demonstrates one possible start-up procedure
that we employed in our convergence studies and our other numerical
tests to follow.  The first step of size $\dt$ from $t_0$ to $t_1$ is
subdivided into substeps in powers of two.  The SSPRK(5,4) scheme
\cite{Spiteri/Ruuth:newclass,kraaijevanger1991} or the SSPRK(10,4)
scheme \cite{ketcheson2008} is used for the first substep, with the
stepsize $\dt^{*}$ chosen small enough so that the local truncation
error of the Runge--Kutta scheme is smaller than the global error of
the TSRK scheme.  Specifically, this can be achieved for an
TSRK method of order $p=5,6,7$ or $8$ by taking
\begin{align} 
  \dt^{*} = \frac{\dt}{2^{\gamma}}, \, \gamma \in \mathbb{Z},
  \quad \text{and} \quad
  (\dt^{*})^5 = A\dt^p = O(\dt^p).  \label{eq:startstepsize}
\end{align}
Subsequent substeps are taken with the TSRK scheme itself, doubling
the stepsizes until reaching $t_1$.  From there, the TSRK scheme
repeatedly advances the solution from $t_n$ to $t_{n+1}$ using
previous step values $u_{n-1}$ and $u_{n}$.

    % Colin learning TikZ:
    % \begin{tikzpicture}[auto]
    %   \node (t0) at (0,0)  [label=below:$t_0$]  {$\bullet$};
    %   \node (t1) at (4,0)  [label=below:$t_1$]  {$\bullet$};
    %   \node (t2) at (8,0)  [label=below:$t_2$]  {$\bullet$};
    %   \node (t3) at (12,0) [label=below:$t_3$]  {$\bullet$};
    %   \node (s1) at (1,0)  []  {$\bullet$};
    %   \node (s2) at (2,0)  []  {$\bullet$};

    %   % \node [red,below] at (t2) {$t_2$};
    %   % \draw [->] (enter t0.east) -- (t1.west);
    %   % \draw [->] (t0) -- (t1);

    %   \draw [->] (t0) [bend right=15] to node [swap] {\small TSRK} (t2);
    %   \draw [->] (t0) [bend left] to node  {\small SSP54} (s1);
    %   \draw [->] (t0) [bend left=45] to node  {\small TSRK} (s2);
    %   \draw [->] (t0) [bend left=80] to node  {\small TSRK} (t1);
    % \end{tikzpicture}
\begin{figure}
  \centering
  \begin{tikzpicture}[scale=0.75]
    \draw [->,semithick] (0,0) -- (16.7,0);

    \foreach \x in {0,1,2,4,8,12,16}
      \draw[gray,very thin] (\x,-0.9) -- (\x, 2.1);
    \foreach \x in {0,1,2,4,8,12,16}
      \filldraw (\x,0) circle (2pt);

    % TODO: this is not the best, if scale is changed above, needs
    % to change here too, better to used named nodes too.
    \draw (4,0) node[below=1ex] {\scriptsize
      $\underbrace{%
        \phantom{\tikz[scale=0.75] \draw (0,0) -- (8,0);}
      }_{\text{first full step of TSRK}}$
    };

    \draw [gray] (14,0) node[below=1ex] {\scriptsize
      $\underbrace{%
        \phantom{\tikz[scale=0.75] \draw (0,0) -- (4,0);}
      }_{\dt}$
    };

    \draw (2,0) node[above=5.5ex] {\scriptsize
      $\overbrace{%
        \phantom{\tikz[scale=0.75] \draw (0,0) -- (4,0);}
      }^{\text{TSRK}}$
    };

    \draw (1,0) node[above=2.5ex] {\scriptsize
      $\overbrace{%
        \phantom{\tikz[scale=0.75] \draw (0,0) -- (2,0);}
      }^{\text{TSRK}}$
    };

    \draw (0.5,0) node[above=-2pt] {\scriptsize
      $\overbrace{%
        \phantom{\tikz[scale=0.75] \draw (0,0) -- (1,0);}
      }^{\text{SSP54}}$
    };

    \draw (0,0) node[below] {\scriptsize $t_{0}$};
    \draw (4,0) node[below] {\scriptsize $t_{1}$};
    \draw (8,0) node[below] {\scriptsize $t_{2}$};
    \draw (12,0) node[below] {\scriptsize $t_{3}$};
    \draw (16,0) node[below] {\scriptsize $t_{4}$};
  \end{tikzpicture}%
  %}  % end centerline
  \caption{One possible startup procedure for SSP TSRK schemes. The
    first step from $t_0$ to $t_1$ is subdivided into substeps
    %in powers of two
    (here there are three substeps of sizes
    $\frac{h}{4}$, $\frac{h}{4}$, and $\frac{h}{2}$).  An SSP
    Runge--Kutta scheme is used for the first substep.  Subsequent
    substeps are taken with the TSRK scheme itself, doubling the
    stepsizes until reaching $t_1$.  We emphasize that the startup
    procedure is not critical for this class of TSRK methods.}
  \label{fig:startup}
\end{figure}
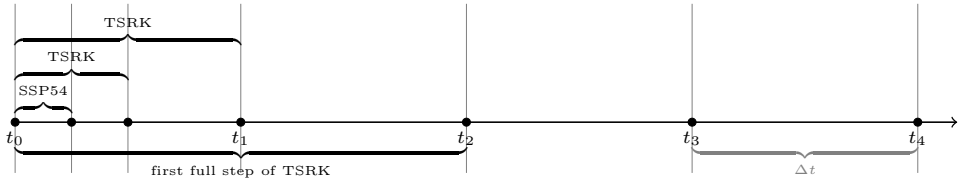

\subsection{Order Verification}

Convergence studies on two ODE test problems confirm that the
SSP TSRK methods achieve their design orders.  The first is the
Dahlquist test problem $u' = \lambda u$, with $u^0 = 1$ and
$\lambda=2$, solved until $t_f=1$.  Figure~\ref{fig:Dahlquist} shows a
sample of TSRK methods achieving their design orders on this problem.
The starting procedure used SSPRK(10,4) with the constant $A$ in
\eqref{eq:startstepsize}
% was set to $\frac{1}{2}$ for orders $p=4,5$
%$10^{-2}$ for $p=6$, and $10^{-3}$ for $p=7,8$.
set respectively to $[\frac{1}{2}, \frac{1}{2}, 10^{-2}, 10^{-3},
10^{-3}]$ for orders $p=4$, $5$, $6$, $7$, and $8$.
%$10^{-2}$ for $p=6$, and $10^{-3}$ for $p=7,8$.

\begin{figure}
  \centerline{%
  \includegraphics[width=0.5\textwidth]{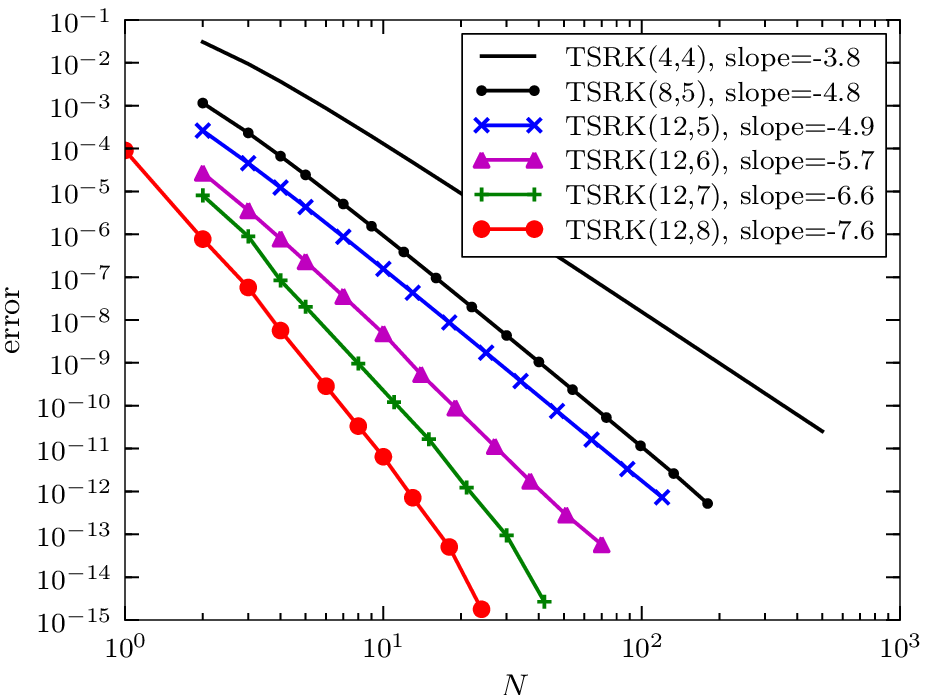}%
  \includegraphics[width=0.5\textwidth]{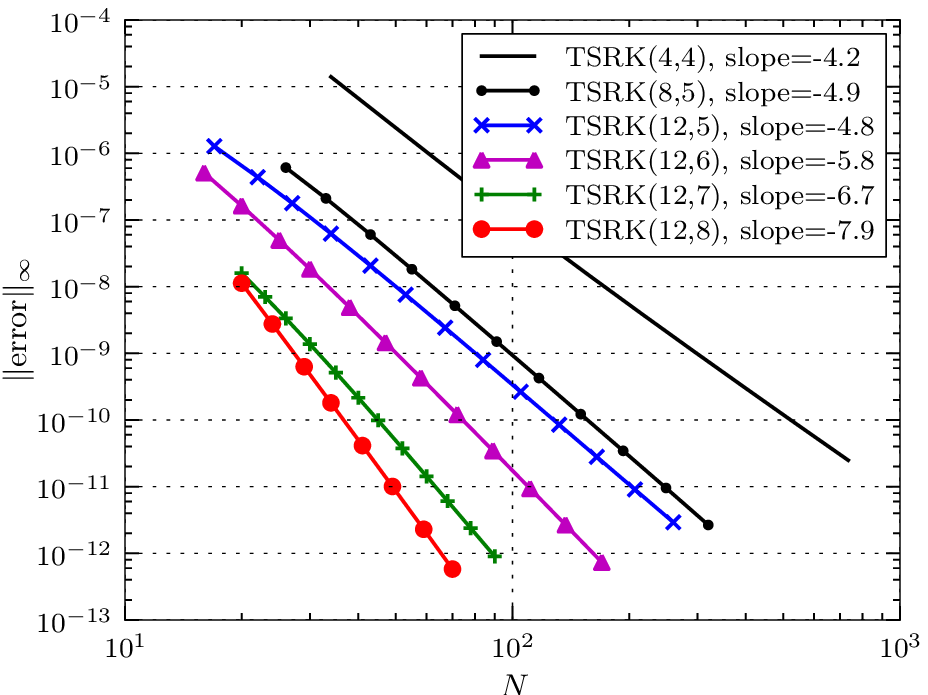}%
  }
  \caption{Convergence results for some TSRK schemes on the Dahlquist
    test problem (left) and van~der Pol problem (right).  
    The slopes of the lines confirm the design orders of the TSRK methods.}
  \label{fig:Dahlquist}
\end{figure}

The nonlinear van der Pol problem (e.g.,
\cite{Macdonald/Gottlieb/Ruuth:DSRK}) can be written as an ODE initial
value problem consisting of two components
%\begin{subequations}
%  \label{eq:vdp}
\begin{align*}
  u_1' &= u_2, \\
  u_2' &= \frac{1}{\epsilon} \left( -u_1 + (1-u_1^2) u_2\right),
\end{align*}
%\end{subequations}
where we use $\epsilon = 0.01$ with corresponding initial conditions
$u^0 = [2; -0.6654321]$ and solve until $t_f=\frac{1}{2}$.  The
starting procedure used SSPRK(10,4) with constant $A=1$ in \eqref{eq:startstepsize}.
Error in the maximum norm is estimated against a highly-accurate
reference solution calculated with \textsc{Matlab}'s ODE45 routine.
Figure~\ref{fig:Dahlquist} shows a sample of the TSRK schemes
achieving their design orders on this problem.

%For the fifth-order two-step Runge--Kutta scheme, we use fourth-order
%SSP RK (e.g., SSP(5,4) or SSP(10,4)) starting methods.

%{\bf todo:} we said we probably don't need a figure or table here.
%But should at least give details of the problem. 

\subsection{High-order WENO}
\label{sec:hiweno}

Weighted essentially non-oscillatory schemes (WENO) 
\cite{harten1987a, harten1987b,jiang1996} 
are finite difference or finite volume schemes
that use linear combination of lower order fluxes 
to obtain a higher order approximations, while ensuring non-oscillatory solutions.  
This is accomplished by using adaptive stencils which approach centered difference 
stencils in smooth regions and one-sided stencils near discontinuities.  
Many WENO methods exist, and the difference between
them is in the computation of the stencil weights.  WENO methods can be constructed to be high order \cite{gerolymos2009, Balsara_Shu_2000:WENO9}.
In   \cite{gerolymos2009}, WENO of up to 17th-order were constructed and  tested numerically.
However, the authors note that in some of their computations the
error was limited by the order of the time
integration, which was relatively low (third-order SSPRK(3,3)).  In
Figure~\ref{fig:high_order_weno}, we reproduce the numerical
experiment of \cite[Fig.~15]{gerolymos2009}, specifically the 2D
linear advection of a sinusoidal initial condition $u_0(x,y) = \sin(\pi(x+y))$, 
in a periodic square using various high-order WENO
methods and our TSRK integrators of order 5, 7 and 8 using 12 stages.
Compared with \cite[Fig.~15]{gerolymos2009}, we note that the error is
no longer dominated by the temporal error.  Thus the higher-order SSP
TSRK schemes allow us to see the behavior of the high-order WENO
spatial discretization schemes.

\begin{figure}
  \centerline{%
  \includegraphics[width=0.5\textwidth]{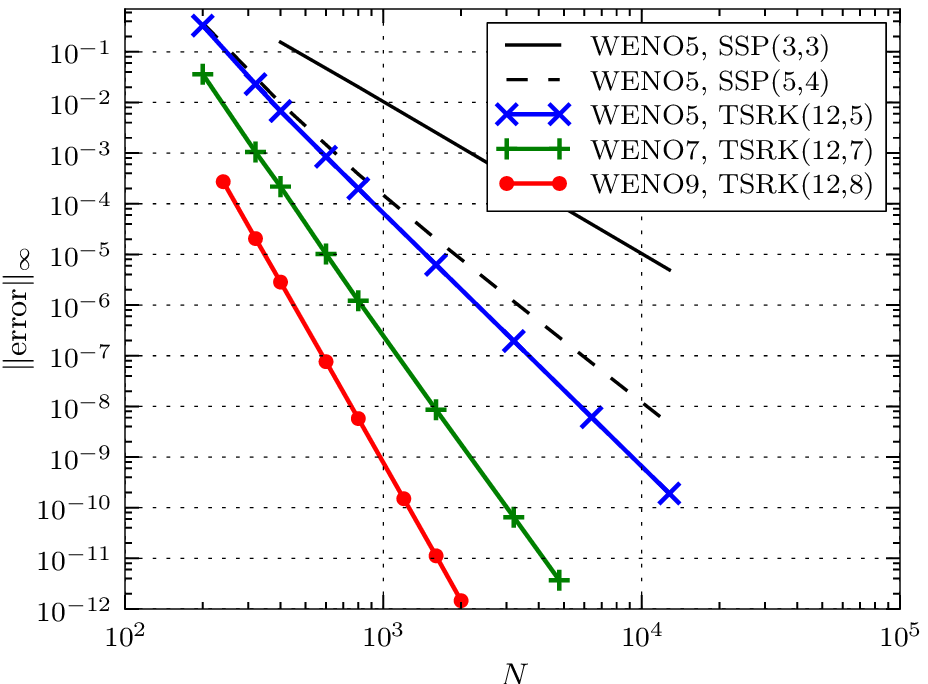}%
  \includegraphics[width=0.5\textwidth]{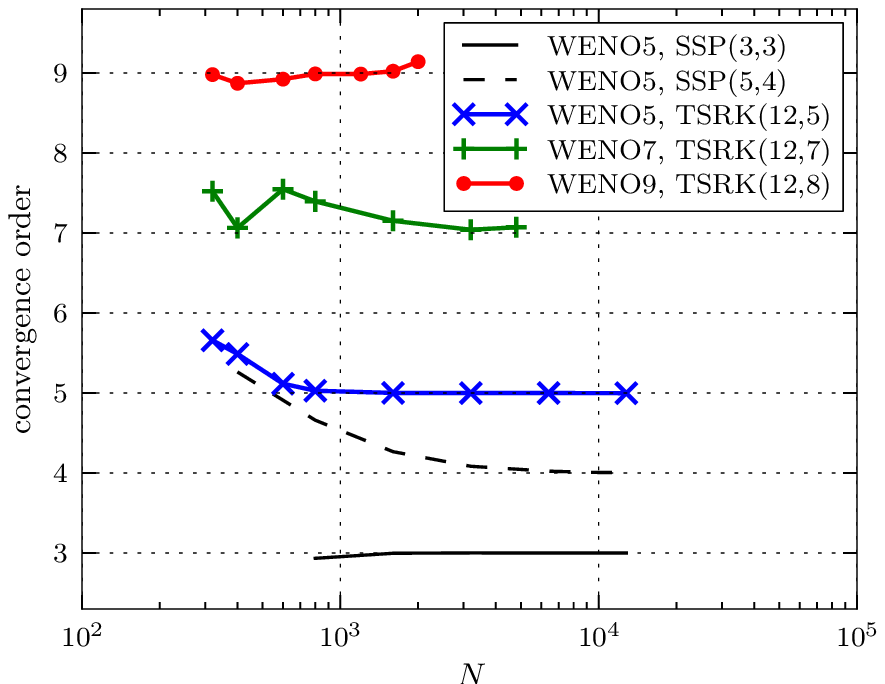}
  }
  \caption{Convergence results for 2D advection using $r$th-order WENO
    discretizations and the TSRK integrators (c.f.,
    \cite[Fig.~15]{gerolymos2009}).  Maximum error versus number of
    spatial grid points in each direction (left).  Observed orders of
    accuracy calculated from these errors (right).  Computed using the
    same parameters as \cite[Fig.~15]{gerolymos2009} (final time
    $t_f=20$, $\dt=0.5\dx$, mapped WENO spatial discretization with
    $p_{\beta} = r$).  Starting procedure as described in
    Section~\ref{sec:startup} using the SSPRK(5,4) scheme for the
    initial substep.}
% p_beta=r
%  WENOM
  \label{fig:high_order_weno}
\end{figure}

\subsection{Buckley--Leverett}

The Buckley--Leverett equation is a model for two-phase flow through
porous media and consists of the conservation law
%\eqref{eq:HCL} with flux function
\begin{equation*}
U_t + f(U)_x = 0, \quad \text{with} \quad
f(U) = \frac{U^2}{U^2 + a(1-U)^2}.
\end{equation*}
We use $a=\frac{1}{3}$ and initial conditions 
%$u(0,x) = \frac{1}{2}H(x-1/2)$
\begin{equation*}
  u(x,0)= \left\{  \begin{array}{ll}
      1 & \mbox{if $x \le \frac{1}{2}$,} \\
      0 & \mbox{otherwise,} \\ \end{array} \right.
\end{equation*}
on $x \in [0,1)$ with periodic boundary conditions.  Our spatial
discretization uses $100$ points and following
\cite{Hundsdorfer/Verwer:2003, 
  ketcheson2009} we use a conservative scheme with Koren limiter.  We
compute the solution until $t_f = \frac{1}{8}$.  For this problem, the
Euler solution is total variation diminishing (TVD) for $\dt \le \DtFE
= 0.0025$ \cite{ketcheson2009}.  As
discussed above, we must also satisfy the SSP time-step restriction
for the starting method.

%The Buckley--Leverett equation is a model for two-phase flow through
%porous media.  As described in
%\cite{Ferracina/Spijker:2007:SDIRKssp,ketcheson2009}, we use it as a
%test case by computing numerical solutions using the various SSP TSRK
%schemes.  We measure the total variation at each time step and make
%sure it is not increasing (i.e., that the solution has the ``TVD''
%property).  We compute the maximal TVD time step as $\dt =
%\sigma_{\text{BL}} \DtFE$ and compare it to $\dt = \sspcoef \DtFE$.

Figure~\ref{fig:bucklev} shows typical solutions using an TSRK scheme
with timestep $\dt = \sigma \DtFE$.  Table~\ref{tab:bucklev_TVD} shows
the maximal TVD time-step sizes, expressed as $\dt =
\sigma_{\text{BL}} \DtFE$, for the Buckley--Leverett test problem.
The results show that the SSP coefficient is a lower bound for what is
observed in practice, confirming the theoretical importance of the SSP
coefficient.

\begin{figure}
  %\centerline{%
  %\includegraphics[width=0.45\textwidth]{figures/bl_tsrk128_s0p9}%
  %\includegraphics[width=0.45\textwidth]{figures/bl_tsrk128_s5p5}%
  %}
  %\centerline{%
  %\includegraphics[width=0.45\textwidth]{figures/bl_tsrk125_s5p2}%
  %\includegraphics[width=0.45\textwidth]{figures/bl_tsrk125_s9}%
  %}
  \centerline{%
  \includegraphics[width=0.45\textwidth]{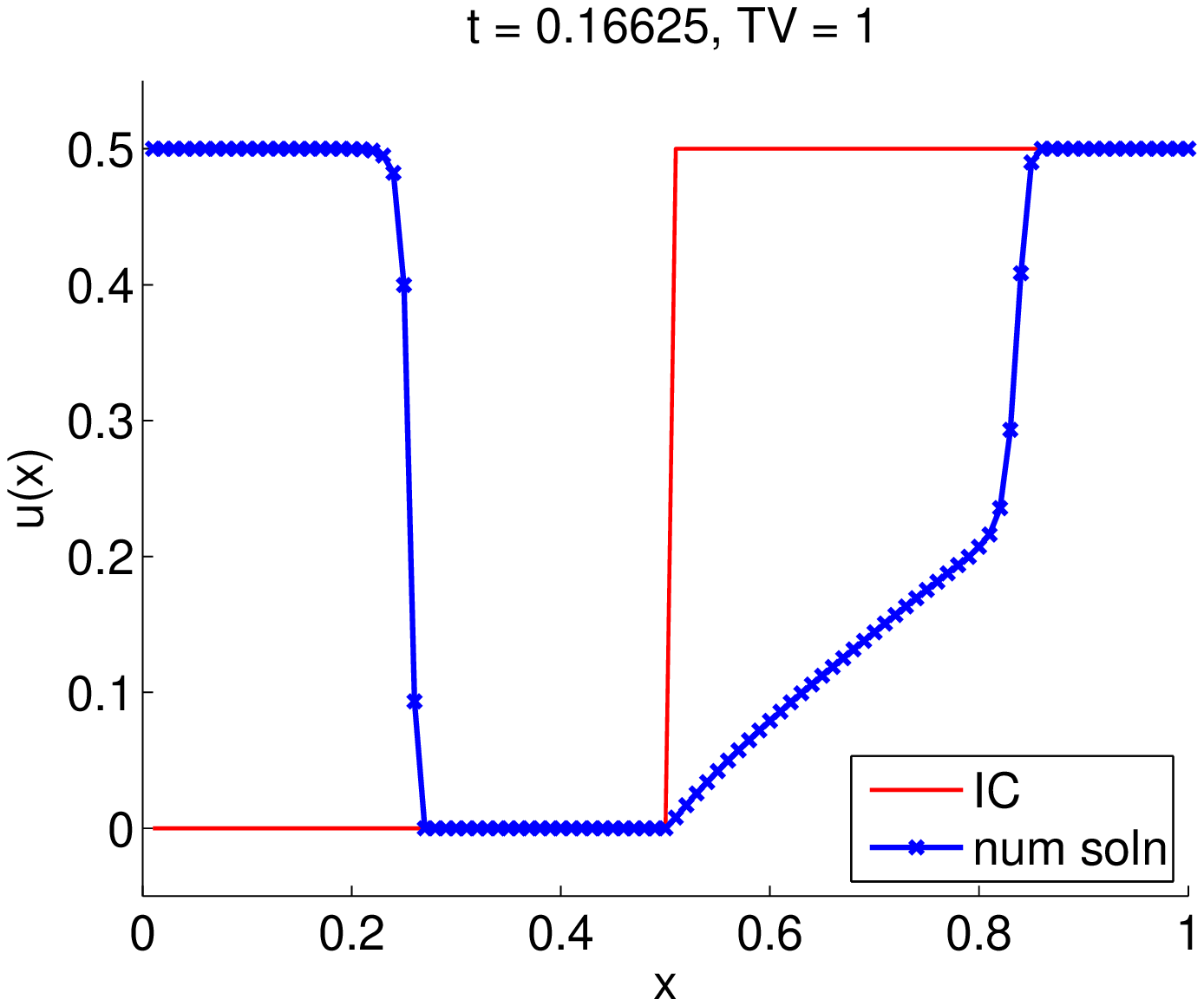}%
  \includegraphics[width=0.45\textwidth]{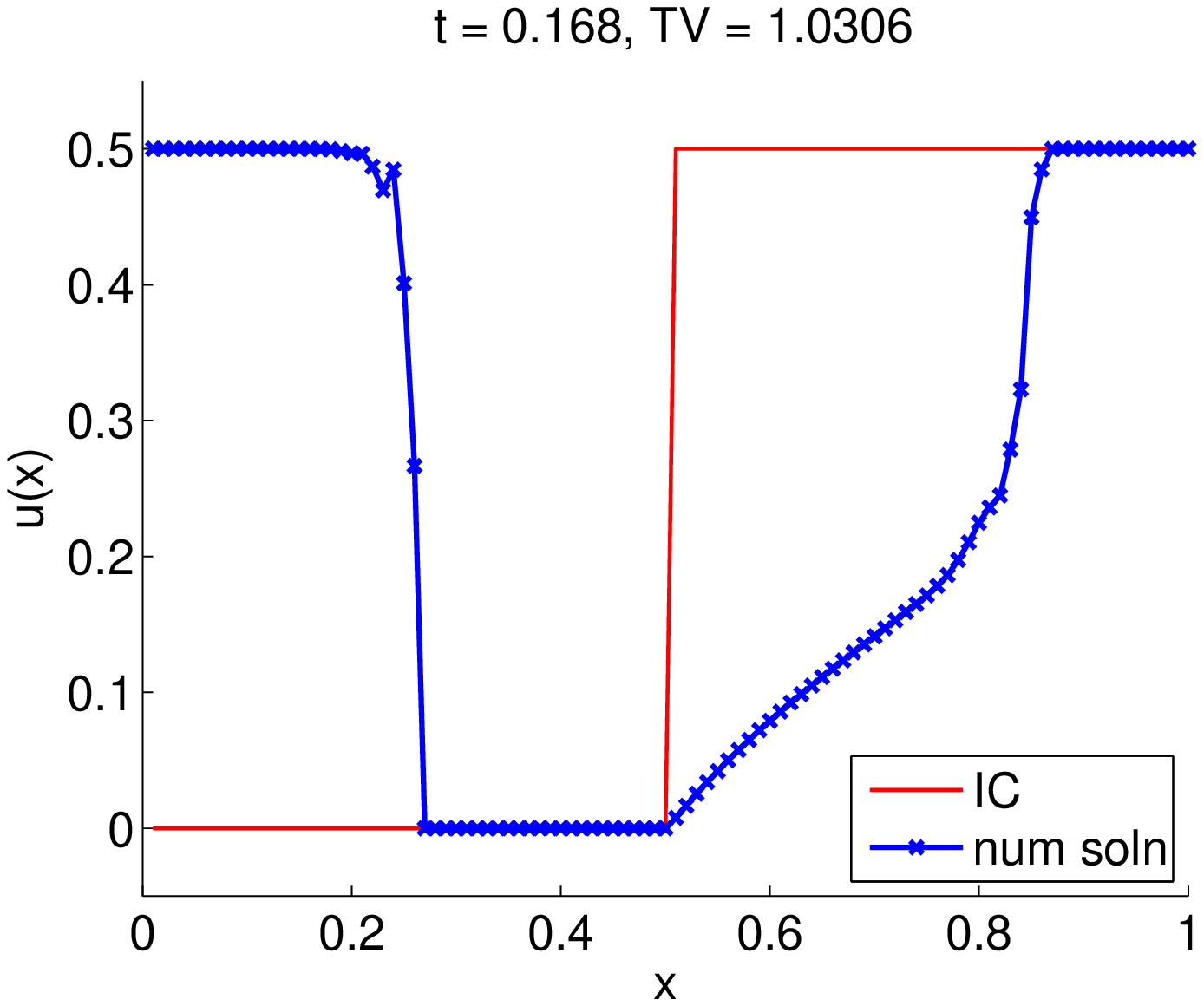}%
  }
  \caption{Two numerical solutions of the Buckley--Leverett test
    problem.  Left: time-step satisfies the SSP time-step restriction
    (TSRK(8,5) using $\dt = 3.5\DtFE$).  Right: time-step
    does not satisfy the restriction ($\dt = 5.6\DtFE$) 
    and visible oscillations have formed, increasing the total variation
    of the solution.}
  \label{fig:bucklev}
\end{figure}

\begin{table}
  \caption{SSP coefficients versus largest time steps exhibiting
    the TVD property ($\dt = \sigma_{\text{BL}} \DtFE$) on
    the Buckley--Leverett example, for
    some of the SSP TSRK($s$,$p$) schemes.  The effective SSP
    coefficient $\ceff$ should be a lower bound for
    $\sigma_{\text{BL}} / s$ and indeed this is observed.
    SSPRK(10,4) \cite{ketcheson2008} is used as the first step
    in the starting procedure.}
  \label{tab:bucklev_TVD}
  \center
  \begin{tabular}{l|cc|cc}
    \hline
    Method & \multicolumn{2}{c|}{theoretical} & \multicolumn{2}{c}{observed} \\
                  & $\sspcoef$ & $\ceff$&  $\sigma_{\text{BL}}$  & $\sigma_{\text{BL}} / s$  \\ \hline
      TSRK(4,4)   & 1.5917   &  0.398     &   2.16 &  0.540   \\
      %TSRK(5,5)   & 1.6213   &  0.324     &   2.53 &  0.506  \\
      TSRK(8,5)   & 3.5794   &  0.447     &   4.41 &  0.551  \\
      TSRK(12,5)  & 5.2675   &  0.439     &   6.97 &  0.581  \\
      TSRK(12,6)  & 4.3838   &  0.365     &   6.80 &  0.567  \\
      TSRK(12,7)  & 2.7659   &  0.231     &   4.86 &  0.405  \\
      TSRK(12,8)  & 0.94155  &  0.0785    &   4.42 &  0.368  \\
      \hline
  \end{tabular}
\end{table}

%% file: conclusions.tex
In this paper we have analyzed the strong stability preserving
property of two-step Runge--Kutta (TSRK) methods.  We find that SSP
TSRK methods have a relatively simple form and that explicit methods
are subject to a maximal order of eight.  We have presented
numerically optimal SSP TSRK methods of order up to this bound of
eight.  These methods overcome the fourth order barrier for (one-step)
SSP Runge--Kutta methods and allow larger SSP coefficients than the
corresponding order multi-step methods.  The discovery of these
methods was facilitated by our formulation of the optimization problem
in an efficient form, aided by simplified order conditions and
constraints on the coefficients derived by using the SSP theory for
general linear methods.  These methods feature favorable storage
properties and are easy to implement and start up, as they do not use
stage values from previous steps.

We show that high-order SSP two-step Runge-Kutta methods are useful
for the time integration of a variety of hyperbolic PDEs, especially
in conjunction with high-order spatial discretizations.  In the case
of a Buckley--Leverett numerical test case, the SSP coefficient of
these methods is confirmed to provide a lower bound for the actual
time-step needed to preserve the total variation diminishing property.

The order conditions and SSP conditions we have derived for these
methods extend in a very simple way to methods with more steps.
Future work will investigate methods with more steps and will further
investigate the use of start-up methods for use with SSP multi-step
Runge--Kutta methods.

%  The restrictive $\ceff=2$ bound for implicit methods is still
%  present.  Doesn't seem to be much gain for implicit multistep RK
%  methods over implicit RK methods, some methods are same in both classes.
%  \bigskip

%  \emph{Chengming Huang}, 2008 ``Strong stability preserving hybrid methods''.
%
%  \emph{Emil Constantinescu \& Adrian Sandu}, 2009, ``Optimal explicit
%  strong-stability-preserving general linear methods''.
%

%% file: tsrk_coefficients.tex
%\subsection{Coefficients of optimal methods}

\setlength{\columnseprule}{0.4pt}
\setlength{\columnsep}{10pt}
%\setlength{\betalticolsep}{0pt}
%\singlespace

% \begin{table}[htb]
%   %\rule{4.5in}{0.4pt}
%   \caption{Coefficients of the optimal explicit
%     5-stage 5th-order SSP TSRK method (Type II)}
%     \label{tbl-tsrk55}
% %    \centering
%     \begin{scriptsize}
%    \begin{multicols*}{3}
% $ \tilde{\theta} =  0.020236082315613$\\
% $ \tilde{d}_0 =    1.000000000000000$\\
% $ \tilde{d}_2 =   0.111755628994224$\\
% $ \tilde{d}_3 =   0.063774568581245$\\
% $ \tilde{d}_5 =   0.061688513153612$\\
% $\eta_0 =  0.031579186621766$\\
% $\eta_1 = 0.237251868533299 $\\
% $\eta_5 =0.638227991820927 $ \\ %\columnbreak \\
% $q_{2,0}=  0.165966644335184$\\
% $q_{3,0}=   0.130982771813100$\\
% $q_{2,1}=   0.722277726670592$\\
% $q_{4,1}=    0.233632131834591$\\
% $q_{5,1}=    0.533135921857798$\\
% $q_{3,2}= 0.805242659605655$\\
% $q_{4,3}= 0.476948026746622$\\
% $q_{5,4}=0.375919004281857$\\
%       \end{multicols*}%}
%     \end{scriptsize}%
%   \rule{4.5in}{0.4pt}
% \end{table}

\begin{table}[htb]
  %\rule{4.5in}{0.4pt}
  \caption{Coefficients of the optimal explicit
  8-stage 5th-order SSP TSRK method (Type II) }
    \label{tbl-tsrk85}
%    \centering
    \begin{scriptsize}
   \begin{multicols*}{3}
$ \tilde{\theta} = 0$\\
$ \tilde{d}_0 =    1.000000000000000$\\
$ \tilde{d}_7 =   0.003674184820260 $\\
$\eta_2 =0.179502832154858$\\
$\eta_3 = 0.073789956884809$\\
$\eta_6 =0.017607159013167$\\
$\eta_8 = 0.729100051947166$\\
$q_{2,0}= 0.085330772947643$\\
$q_{3,0}=0.058121281984411$\\
$q_{7,0}= 0.020705281786630$\\
$q_{8,0}=   0.008506650138784$\\
$q_{2,1}=   0.914669227052357$\\
$q_{4,1}=   0.036365639242841$\\
$q_{5,1}=   0.491214340660555$\\
 $q_{6,1}=  0.566135231631241$\\
$q_{7,1}=   0.091646079651566$\\
 $q_{8,1}=  0.110261531523242$\\
$q_{3,2}=   0.941878718015589$\\
$q_{8,2}=   0.030113037742445$\\
$q_{4,3}=   0.802870131352638$\\
$q_{5,4}= 0.508785659339445$\\
$q_{6,5}= 0.433864768368758$\\
$q_{7,6}= 0.883974453741544$\\
 $q_{8,7}= 0.851118780595529$\\
      \end{multicols*}%}
    \end{scriptsize}%
  \rule{4.5in}{0.4pt}
\end{table}

\begin{table}[htb]
  %\rule{4.5in}{0.4pt}
  \caption{Coefficients of the optimal explicit
    12-stage 5th-order SSP TSRK method (Type II)} 
    \label{tbl-tsrk125}
%    \centering
    \begin{scriptsize}
   \begin{multicols*}{3}
$ \tilde{\theta} = 0 $ \\
$ \tilde{d}_0 =   1$ \\
 $ \eta_1 =  0.010869478269914 $  \\
 $ \eta_6=0.252584630617780 $  \\
 $ \eta_{10} = 0.328029300816831 $  \\
$ \eta_{12} = 0.408516590295475 $  \\
 $ q_{2,0} =   0.037442206073461 $  \\
$ q_{3,0} = 0.004990369159650 $  \\ 
$ q_{2,1}= 0.962557793926539 $  \\
$ q_{6,1} =  0.041456384663457 $  \\
$ q_{7,1}=   0.893102584263455 $  \\
$ q_{9,1}= 0.103110842229401 $  \\
$ q_{10,1}=   0.109219062395598 $  \\ %\columnbreak  \\
$ q_{11,1}=   0.069771767766966 $  \\
$ q_{12,1}=   0.050213434903531 $  \\
$ q_{3,2}= 0.750941165462252 $  \\
$ q_{4,3}=0.816192058725826  $  \\
$ q_{5,4}=0.881400968167496  $  \\
 $ q_{6,5}=  0.897622496599848  $  \\
 $ q_{7,6}=0.106897415736545  $  \\
  $ q_{8,6}=0.197331844351083 $     \\
$ q_{8,7} = 0.748110262498258  $  \\
$ q_{9,8} =0.864072067200705 $  \\
$ q_{10,9} = 0.890780937604403  $  \\
$ q_{11,10} =  0.928630488244921 $  \\
$ q_{12,11} = 0.949786565096469  $  \\
      \end{multicols*}%}
    \end{scriptsize}
    \rule{4.5in}{0.4pt}
\end{table}

\begin{table}[htb]
  %\rule{4.5in}{0.4pt}
  \caption{Coefficients of the optimal explicit
    12-stage 6th-order SSP TSRK method (Type II)} 
    \label{tbl-tsrk126}
%    \centering
    \begin{scriptsize}
  \begin{multicols*}{3}
$ \tilde{\theta} = 2.455884612148108e-04$ \\
$ \tilde{d}_0 =   1$ \\
$ \tilde{d}_10 = 0.000534877909816$\\
$ q_{2,0} = 0.030262100443273$\\
$ q_{2,1} =  0.664746114331100$\\
$ q_{6,1} =  0.656374628865518$\\
$ q_{7,1} =    0.210836921275170$\\
$ q_{9,1} =    0.066235890301163$\\
$ q_{10,1} =   0.076611491217295$\\
$ q_{12,1} =    0.016496364995214$\\
$ q_{3,2} =  0.590319496200531$\\
$ q_{4,3} =  0.729376762034313$\\
$ q_{5,4} =    0.826687833242084 $\\
$ q_{10,4} =   0.091956261008213$\\
$ q_{11,4} =   0.135742974049075$\\
$ q_{6,5} =0.267480130553594$\\
$ q_{11,5} =0.269086406273540$\\
$ q_{12,5} =   0.344231433411227$\\
$ q_{7,6} =    0.650991182223416$\\
$ q_{12,6} =   0.017516154376138$\\
$ q_{8,7} =0.873267220579217 $\\
$ q_{9,8} =0.877348047199139$\\
$ q_{10,9} = 0.822483564557728$\\
$ q_{11,10} = 0.587217894186976$\\
$ q_{12,11} = 0.621756047217421$\\    
 $\eta_1=  0.012523410805564$\\
 $\eta_6 = 0.094203091821030$\\
  $\eta_9 = 0.318700620499891$\\
   $\eta_{10} =0.107955864652328$\\
 $\eta_{12} =   0.456039783326905$\\
       \end{multicols*}%}
     \end{scriptsize}
     \rule{4.5in}{0.4pt}
\end{table}

\begin{table}[htb]
  %\rule{4.5in}{0.4pt}
  \caption{Coefficients of the optimal explicit 12-stage
    7th-order SSP TSRK method (Type II)} 
    \label{tbl-tsrk127}
%    \centering
    \begin{scriptsize}
    \begin{multicols*}{3}
$ \tilde{\theta} =  1.040248277612947e-04$\\
$ \tilde{d}_0 =     1.000000000000000$ \\
$ \tilde{d}_2 =    0.003229110378701$ \\
$ \tilde{d}_4 =    0.006337974349692$ \\
$ \tilde{d}_5 =    0.002497954201566$ \\
$ \tilde{d}_8 =    0.017328228771149$ \\
$ \tilde{d}_12 =   0.000520256250682$ \\
$ \eta_{0} = 0.000515717568412$\\
$ \eta_1 =   0.040472655980253 $  \\
$ \eta_6 =   0.081167924336040 $  \\
$ \eta_7 = 0.238308176460039 $  \\
$ \eta_8=   0.032690786323542 $  \\
$ \eta_{12}=0.547467490509490 $  \\
$q_{2,0}=  0.147321824258074$\\
$ q_{2,1} =   0.849449065363225 $  \\
$ q_{3,1} =   0.120943274105256 $  \\
$ q_{4,1} =   0.368587879161520 $  \\
$ q_{5,1} =   0.222052624372191 $  \\
$ q_{6,1} =   0.137403913798966 $  \\
$ q_{7,1} =   0.146278214690851 $  \\
$ q_{8,1} =   0.444640119039330 $  \\
$ q_{9,1} =   0.143808624107155 $  \\
$ q_{10,1} =   0.102844296820036 $ \\ %\columnbreak  \\
$ q_{11,1} =   0.071911085489036 $  \\
$ q_{12,1} =   0.057306282668522 $  \\
$ q_{3,2} = 0.433019948758255 $  \\
$ q_{7,2}=  0.014863996841828 $  \\
$q_{9,2}  =  0.026942009774408 $  \\
$ q_{4,3}=  0.166320497215237 $     \\
$ q_{10,3}=   0.032851385162085 $  \\
 $ q_{5,4}=0.343703780759466 $  \\
 $ q_{6,5}=  0.519758489994316 $  \\
$ q_{7,6}= 0.598177722195673 $  \\
$ q_{8,7}=   0.488244475584515 $  \\
$ q_{10,7}=   0.356898323452469 $  \\
$ q_{11,7}=   0.508453150788232 $  \\
$ q_{12,7}=   0.496859299069734 $  \\
$ q_{9,8}=    0.704865150213419 $  \\
$ q_{10,9}=  0.409241038172241 $  \\
$ q_{11,10}=  0.327005955932695 $  \\
$ q_{12,11}= 0.364647377606582 $  \\
      \end{multicols*}%}
    \end{scriptsize}
  \rule{4.5in}{0.4pt}
\end{table}

\begin{table}[htb]
  %\rule{4.5in}{0.4pt}
  \caption{Coefficients of the optimal explicit 12-stage 8th-order SSP TSRK method (Type II)} 
    \label{tbl-tsrk128}
%    \centering
    \begin{scriptsize}
    \begin{multicols*}{3}
    $ \tilde{\theta} = 4.796147528566197e-05 $  \\
    $ \tilde{d}_0  =   1.000000000000000$  \\
    $ \tilde{d}_2  =   0.036513886685777$  \\
   $ \tilde{d}_4 =    0.004205435886220$  \\
    $ \tilde{d}_5  =   0.000457751617285$  \\
    $ \tilde{d}_7  =   0.007407526543898$  \\
   $ \tilde{d}_8 =    0.000486094553850$  \\
   $ \eta_1=   0.033190060418244 $  \\
$ \eta_2= 0.001567085177702 $  \\
$ \eta_3=   0.014033053074861 $  \\
$ \eta_4=   0.017979737866822 $  \\
$ \eta_5= 0.094582502432986 $  \\
$ \eta_6 = 0.082918042281378 $  \\
$ \eta_7= 0.020622633348484 $  \\
$ \eta_8=  0.033521998905243 $  \\
$ \eta_9=0.092066893962539 $  \\
$ \eta_{10}=  0.076089630105122 $  \\
$ \eta_{11}=0.070505470986376 $  \\
$ \eta_{12}=0.072975312278165 $  \\
$ q_{2,0}=                0.017683145596548$  \\
$ q_{3,0}=    0.001154189099465$  \\
$ q_{6,0}=    0.000065395819685$  \\
$ q_{9,0}=    0.000042696255773$  \\
$ q_{11,0}=    0.000116117869841$  \\
$ q_{12,0}=    0.000019430720566$ \\ %\columnbreak \\
$ q_{2,1}=   0.154785324942633 $  \\
$ q_{4,1}=   0.113729301017461 $  \\
$ q_{5,1}=   0.061188134340758 $  \\
$ q_{6,1}=   0.068824803789446 $  \\
$ q_{7,1}=   0.133098034326412 $  \\
$ q_{8,1}=   0.080582670156691 $  \\
$ q_{9,1}=   0.038242841051944 $  \\
$ q_{10,1}=   0.071728403470890 $  \\
$ q_{11,1}=   0.053869626312442 $  \\
$ q_{12,1}=   0.009079504342639 $  \\
$ q_{3,2}=   0.200161251441789 $  \\
$ q_{6,2}=   0.008642531617482 $  \\
$ q_{4,3}=   0.057780552515458 $  \\
$ q_{9,3}=   0.029907847389714 $  \\
$ q_{5,4}=   0.165254103192244 $  \\
$ q_{7,4}=   0.005039627904425 $  \\
$ q_{8,4}=   0.069726774932478 $  \\
$ q_{9,4}=   0.022904196667572 $  \\
$ q_{12,4}=   0.130730221736770 $  \\
 $ q_{6,5}=   0.229847794524568 $  \\
 $ q_{9,5}=     0.095367316002296 $  \\
$ q_{7,6}=    0.252990567222936 $  \\
$ q_{9,6}=   0.176462398918299 $  \\
$ q_{10,6}=   0.281349762794588 $  \\
$ q_{11,6}=   0.327578464731509 $  \\
$ q_{12,6}=   0.149446805276484 $  \\
$ q_{8,7}=0.324486261336648 $  \\
$ q_{9,8}= 0.120659479468128 $  \\
$ q_{10,9}=0.166819833904944 $  \\
$ q_{11,10}= 0.157699899495506 $  \\
$ q_{12,11}= 0.314802533082027 $   \\
      \end{multicols*}%}
    \end{scriptsize}
    \rule{4.5in}{0.4pt}
\end{table}